\documentclass[10pt]{article}

\usepackage{latexsym}
\usepackage{amsmath,exscale}
\usepackage{amssymb}
\usepackage{amsfonts}
\usepackage{amsthm}
\usepackage{amscd}
\usepackage{epsfig}
\usepackage{verbatim}
\usepackage{fancybox}
\usepackage{moreverb}
\usepackage{graphicx}
\usepackage{psfrag}
\usepackage[all]{xy}
\usepackage[toc,page]{appendix}

\usepackage[applemac]{inputenc}

\usepackage{color}
\definecolor{marin}{rgb}   {0.,   0.3,   0.7} 
\definecolor{rouge}{rgb}   {0.8,   0.,   0.} 
\definecolor{sepia}{rgb}   {0.8,   0.5,   0.} 
\usepackage[colorlinks,citecolor=marin,linkcolor=rouge,
            bookmarksopen,
            bookmarksnumbered
           ]{hyperref}

\newcommand{\R}{{\mathbb R}}

\newcommand{\D}{{\partial}}

\newcommand{\ucar}{\text{\scriptsize $\sqcup$\hspace{-6pt}$\boldsymbol\sqcup$\hspace{-5pt}$\sqcup$}}
\newcommand{\Nabla}{{\nabla\hspace{-2pt}}}

\newtheorem{theorem}{Theorem}[section]
\newtheorem{lemma}[theorem]{Lemma}

\newtheorem{property}[theorem]{Property}

\newtheorem{remark}[theorem]{Remark}

\def\commutatif{\ar@{}[rd]|{\circlearrowleft}}

\numberwithin{equation}{section}

\title{Two-Scale Macro-Micro decomposition of the Vlasov equation with a strong magnetic field}

\author{Nicolas CROUSEILLES\thanks{Inria Rennes-Bretagne Atlantique, IPSO Project \& IRMAR (UMR CNRS 6625) Universit\'e de Rennes 1.}
\and
Emmanuel FR\'ENOD\thanks{Universit\'e Europ\'eenne de Bretagne, LMBA (UMR CNRS 6205), Universit\'e de Bretagne-Sud \& Inria Nancy-Grand Est, CALVI Project  \& IRMA (UMR CNRS 7501), Universit\'e de Strasbourg.}
\and
Sever A. HIRSTOAGA\thanks{Inria Nancy-Grand Est, CALVI Project \& IRMA  (UMR CNRS 7501), Universit\'e de Strasbourg.}
\and
Alexandre MOUTON\thanks{CNRS \& Laboratoire Painlev\'e (UMR CNRS 8524) Universit\'e de Lille 1.}
}
\date{}

\begin{document}

\maketitle

\begin{abstract}
In this paper, we build a Two-Scale Macro-Micro decomposition of the Vlasov equation with a strong magnetic field. This consists in writing the solution of this equation as a sum of two oscillating functions with circumscribed oscillations. The first of these functions has a shape which is close to the shape of the Two-Scale limit of the solution and the second one is a correction built to offset this imposed shape. The aim of such a decomposition is to be the starting point for the construction of Two-Scale Asymptotic-Preserving Schemes.
\end{abstract}

\section{Introduction}
\setcounter{equation}{0}

The goal of this paper is to make the first step towards the setting out of a new class of numerical methods: the \textit{Two-Scale Asymptotic-Preserving Schemes} or TSAPS. We intend to use these methods in order to treat problems involving the following two numerical difficulties: first, dealing efficiently on long time scales with solutions having high frequency oscillations (the \textit{Two-Scale} approach) and second, an accurate and stable managing of the transition between different regimes, using a unified model (the \textit{Asymptotic-Preserving} approach). These problems naturally occur in solving a Vlasov equation implying some small parameter, for the modelling of the dynamics of charged particles in the presence of a strong magnetic field. More precisely, we are interested in developing a model whose discretization will be able to simulate both, the regime when the parameter is not small (\textit{i.e.} the magnetic field is not large), and the limit regime obtained when the parameter is small. The discrete scheme will automatically shift from one regime to the other. In addition, in the limit of the small parameter, the particles mean behaviour will be efficiently described by the Two-Scale limit while the micro behaviour will be expressed by some correctors.\\~\\
\indent  In the framework of a large magnetic field modelization, we  refer to \cite{Bostan2007,Frenod-Mouton,Frenod-Sonnen_FLR,Golse-Saint-Raymond_1999, Golse-Saint-Raymond_2003,Han-Kwan} and the references therein as previous works about different regimes, obtained by taking limits of Vlasov or Vlasov-Poisson equations corresponding to small parameters. Then, we refer to \cite{Ailliot-Frenod-Monbet, Allaire, Frenod-Mouton-Sonnen, Frenod-Salvarani-Sonnen,Mouton, Mouton_PhD, Nguetseng} for a detailed description of Two-Scale numerical methods. As for the Asymptotic-Preserving schemes, although there is an abundant literature about, we first cite the classical paper \cite{Jin} and then \cite{Bennoune, Crouseilles, filbet-jin, klar, Lemou-Mieussens} which treat problems close to ours, without considering the two-scale aspect.\\
\indent Next we start to detail the method, by setting a generic and formal problem
\begin{equation} \label{generic_Eu_eps}
E^{\epsilon}\,u^{\epsilon} = 0 \, ,
\end{equation}
in which $E^{\epsilon}$ is a partial differential operator depending on a parameter $\epsilon$ which induces oscillations of typical size $\epsilon$ in the solution $u^{\epsilon} = u^{\epsilon}(z)$. In the case where $\epsilon$ is uniformly small over the domain on which (\ref{generic_Eu_eps}) is posed, a Two-Scale approach may give good results. To summarize, this approach consists in noticing that, in some sense,
\begin{equation} \label{generic_Eu_TS_rebuild}
u^{\epsilon}(z) \approx U(z,\frac{z}{\epsilon}) \, ,
\end{equation}
where $U = U(z,\zeta)$ is the Two-Scale limit of $u^{\epsilon}$ (Two-Scale convergence goes back to \cite{Nguetseng}), and that $U$ is the solution of a well-posed problem of the form
\begin{equation} \label{qeneric_Eu_TS}
\mathcal{E}\,U = 0 \, ,
\end{equation}
where $\mathcal{E}$ does not induce high frequency oscillations. Building a Two-Scale numerical method consists then in solving, via an usual numerical method, model (\ref{qeneric_Eu_TS}) and rebuilding $u^{\epsilon}$ using an algorithm based on (\ref{generic_Eu_TS_rebuild}). \\
\indent In the case where $\epsilon$ is not a uniform parameter over the domain, in the sense that it may be small in some sub-region and of order 1 in other ones, Asymptotic-Preserving schemes are intended to be used, since they allow to simulate both sub-regions with the same method. The methodology for these numerical schemes is based on the fact that $u^{\epsilon}$ may converge, in a strong or a weak topology, towards a limit $u$ which is solution of
\begin{equation} \label{generic_Eu_limit}
E\,u = 0 \, .
\end{equation}
Then, an Asymptotic-Preserving scheme consists in solving 
\begin{equation} \label{generic_Eu_eps_AP}
E_{\Delta z}^{\epsilon}\,u_{\Delta z}^{\epsilon} = 0 \, ,
\end{equation}
such that
\begin{equation} \label{generic_eps_consistency_AP}
u_{\Delta z}^{\epsilon} \longrightarrow u^{\epsilon} \, , \qquad \textnormal{as $\Delta z \to 0$,}
\end{equation}
and
\begin{equation} \label{generic_cv_eps_AP}
u_{\Delta z}^{\epsilon} \longrightarrow u_{\Delta z} \, , \qquad \textnormal{as $\epsilon \to 0$,}
\end{equation}
where $u_{\Delta z}$ is a numerical approximation of $u$, meaning that
\begin{equation} \label{generic_0_consistency_AP}
u_{\Delta z} \longrightarrow u \, , \qquad \textnormal{as $\Delta z \to 0$,}
\end{equation}
and is the solution of
\begin{equation} \label{generic_Eu_0_AP}
E_{\Delta z}\, u_{\Delta z} = 0 \, ,
\end{equation}
where $E_{\Delta z}$ is a numerical operator. Notice that 
(\ref{generic_eps_consistency_AP})-(\ref{generic_cv_eps_AP})-(\ref{generic_0_consistency_AP}) is called the \textit{Asymptotic-Preserving property}. These convergence results may be summarized in the following diagram: \\
\begin{equation} \label{AP_diagram}
\xymatrix{
\textnormal{
\begin{tabular}{c}
$u^{\epsilon}$ solution of \\
$E^{\epsilon}\,u^{\epsilon} = 0$
\end{tabular}
}
\ar[rr]^{\epsilon \, \to \, 0}
& & 
\textnormal{
\begin{tabular}{c}
$u$ solution of \\
$E\,u = 0$
\end{tabular}
}
\\ 
& & \\
\textnormal{
\begin{tabular}{c}
$u_{\Delta z}^{\epsilon}$ solution of \\
$E_{\Delta z}^{\epsilon}\,u_{\Delta z}^{\epsilon} = 0$
\end{tabular}
} 
\ar[rr]^{\epsilon \, \to \, 0}
\ar[uu]_{\Delta z \, \to \, 0}
& & 
\textnormal{
\begin{tabular}{c}
$u_{\Delta z}$ solution of \\
$E_{\Delta z}\,u_{\Delta z} = 0$
\end{tabular}
} 
\ar[uu]^{\Delta z \, \to \, 0}
}
\end{equation}

In the case when $u^{\epsilon}$ contains high frequency oscillations with high amplitude when $\epsilon$ is small, solving (\ref{generic_Eu_eps_AP}) can be time consuming to capture the oscillations. In this case, the convergence
\begin{equation} \label{generic_weak_cv}
u^{\epsilon} \to u \, ,
\end{equation}
occurs in a weak sense only. Besides this, the Two-Scale limit $U = U(z,\zeta)$ may describe very well the oscillations of $u^{\epsilon}$ in the sense that the convergence
\begin{equation}
U(z,\frac{z}{\epsilon})-u^{\epsilon}(z) \to 0 \, ,
\end{equation}
occurs in a stronger sense than (\ref{generic_weak_cv}). This is particularly the case when $E^{\epsilon}$ generates pseudo-periodic oscillations, essentially at a sole period, of size $\epsilon$ in the solution $u^{\epsilon}$ when $\epsilon$ becomes small. Moreover, $u$ and $U$ are linked by a relation of the kind
\begin{equation}
u(z) = \frac{1}{|\mathcal{Z}|}\int_{\mathcal{Z}} U(z,\zeta)\,d\zeta \, ,
\end{equation}
where $\mathcal{Z}$ is the cell such that $\zeta \mapsto U(z,\zeta)$ is $\mathcal{Z}$-periodic for any $z$, and $|\mathcal{Z}|$ is its measure. \\~\\
\indent In this case, we propose to add a layer in the diagram (\ref{AP_diagram}) which summarizes the ideas on which Two-Scale Asymptotic-Preserving Scheme are based:
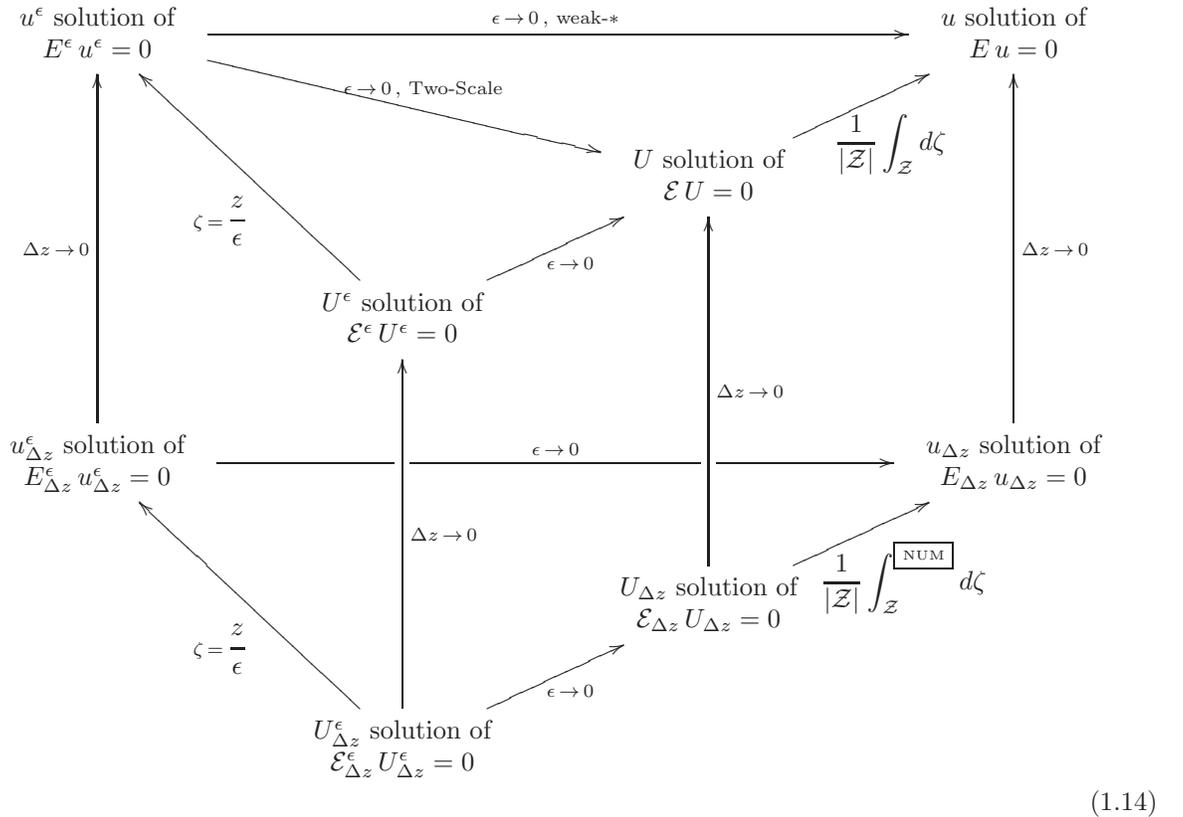
\begin{equation} \label{diagram_TSAPS}
\xymatrix{
\textnormal{
\begin{tabular}{c}
$u^{\epsilon}$ solution of \\ $E^{\epsilon}\,u^{\epsilon} = 0$
\end{tabular}
}
\ar[rrr]^{\epsilon \, \to \, 0 \, , \textnormal{ weak-$*$}} \ar[drr]^{\epsilon \, \to\,0\,, \textnormal{ Two-Scale}} & & & \textnormal{
\begin{tabular}{c}
$u$ solution of \\ $E\,u=0$
\end{tabular}
} \\
& & \textnormal{
\begin{tabular}{c}
$U$ solution of \\ $\mathcal{E}\,U = 0$
\end{tabular}
} \ar[ru]_{\displaystyle \frac{1}{|\mathcal{Z}|}\int_{\mathcal{Z}}d\zeta} & \\
& \textnormal{
\begin{tabular}{c}
$U^{\epsilon}$ solution of \\ $\mathcal{E}^{\epsilon}\,U^{\epsilon} = 0$
\end{tabular}
} \ar[luu]^{\zeta\,=\,\cfrac{z}{\epsilon}} \ar[ru]_{\epsilon\,\to\,0} & & \\
\textnormal{
\begin{tabular}{c}
$u_{\Delta z}^{\epsilon}$ solution of \\ $E_{\Delta z}^{\epsilon} \,u_{\Delta z}^{\epsilon} = 0$
\end{tabular}
} \ar[uuu]^{\Delta z \, \to\, 0} \ar@{->}'[r]
'[rr]^{\epsilon\,\to\,0}
[rrr] & & & \textnormal{
\begin{tabular}{c}
$u_{\Delta z}$ solution of \\ $E_{\Delta z}\,u_{\Delta z} = 0$
\end{tabular}
} \ar[uuu]_{\Delta z \, \to\,0} \\
& & \textnormal{
\begin{tabular}{c}
$U_{\Delta z}$ solution of \\ $\mathcal{E}_{\Delta z}\,U_{\Delta z} = 0$
\end{tabular}
} \ar[uuu]_{\Delta z \, \to\, 0} \ar[ru]_{\displaystyle \frac{1}{|\mathcal{Z}|}\int_{\mathcal{Z}}^\text{\tiny \rm \fbox{NUM}}d\zeta} & \\
& \textnormal{
\begin{tabular}{c}
$U_{\Delta z}^{\epsilon}$ solution of \\ $\mathcal{E}_{\Delta z}^{\epsilon} \, U_{\Delta z}^{\epsilon} = 0$
\end{tabular}
} \ar[luu]^{\zeta\,=\,\cfrac{z}{\epsilon}} \ar[uuu]_{\Delta z\,\to\,0} \ar[ru]_{\epsilon\,\to\,0} & & 
}
\end{equation}
\indent Diagram (\ref{diagram_TSAPS}) may be looked at as a prism. At the bottom of the front edge stands the TSAPS and the back rectangle is made of the Asymptotic-Preserving scheme diagram of (\ref{AP_diagram}). The intermediate layer gathers the Two-Scale limit problem, at the top, and the Two-Scale numerical method, at the bottom.\\
\indent The solution $U_{\Delta z}^{\epsilon}$ computed with the TSAPS must be close to $u^{\epsilon}$ when it is computed in $\zeta = \frac{z}{\epsilon}$ and when $\Delta z$ converges to 0. This is the first part of the Two-Scale Asymptotic-Preserving property and it is symbolized by the arrow going from $U_{\Delta z}^{\epsilon}$ to $u_{\Delta z}^{\epsilon}$ and by the one linking $u_{\Delta z}^{\epsilon}$ and $u^{\epsilon}$. \\
The second part of the Two-Scale Asymptotic-Preserving property is symbolized by the right part of the diagram (\ref{diagram_TSAPS}) : $U_{\Delta z}^{\epsilon}$ must be close to the numerical approximation $U_{\Delta z}$ of the Two-Scale limit of $u_{\epsilon}$ when $\epsilon \to 0$ and a numerical integration (symbolized by $\displaystyle \frac{1}{|\mathcal{Z}|}\int_{\mathcal{Z}}^\text{\tiny \rm \fbox{NUM}} d\zeta$) of $U_{\Delta z}$ needs to recover an approximation $u_{\Delta z}$ of the weak-$*$ limit $u$ of $u^{\epsilon}$.\\
\indent We will deduce the Two-Scale Asymptotic-Preserving model on the base of a Two-Scale Macro-Micro decomposition that yields an equation which stands at the top of the front edge. Arrows linking $U^{\epsilon}$ and $u^{\epsilon}$, $U^{\epsilon}$ and $U$ and $u$ may be seen as the continuous Two-Scale Asymptotic-Preserving property. They are constraints for the Macro-Micro decomposition. We refer to \cite{Bennoune, Crouseilles, Jin-Shi,Lemou-Mieussens} for the use of Macro-Micro decomposition in order to design an Asymptotic-Preserving scheme. \\~\\

\indent The present paper implements the Two-Scale Macro-Micro decomposition in the case of the linear Vlasov equation with a strong magnetic field. This equation is the following:
\begin{equation} \label{Vlasov_rescaled_intro}
\left\{
\begin{array}{l}
\cfrac{\D f^{\epsilon}}{\D t} + \mathbf{v} \cdot \Nabla_{\mathbf{x}} f^{\epsilon} + \Big( \mathbf{E}^{\epsilon} + \mathbf{v} \times \big( \mathbf{B}^{\epsilon} + \cfrac{\mathcal{M}}{\epsilon} \big) \Big) \cdot \Nabla_{\mathbf{v}} f^{\epsilon} = 0 \, , \\ \\
f^{\epsilon}(t=0,\mathbf{x},\mathbf{v}) = f_{0}(\mathbf{x},\mathbf{v}) \, ,
\end{array}
\right.
\end{equation}
where $f^{\epsilon} = f^{\epsilon}(t,\mathbf{x},\mathbf{v})$ is the distribution function of an ion gaz submitted to external electric and magnetic field $\mathbf{E}^{\epsilon} = \mathbf{E}^{\epsilon}(t,\mathbf{x})$ and $\mathbf{B}^{\epsilon} = \mathbf{B}^{\epsilon}(t,\mathbf{x})$, where $t \in \R^{+}$, $\mathbf{x} \in \R^{3}$, $\mathbf{v} \in \R^{3}$, and where $\mathcal{M} = 2\pi\,\mathbf{e}_{1}$, \textit{i.e.} the first vector of the canonical base $(\mathbf{e}_{1},\mathbf{e}_{2},\mathbf{e}_{3})$ of $\R^{3}$. \\

\indent The paper is organized as follows: in Section \ref{SecTSC}, we recall the results of Fr\'enod \& Sonnendr\"ucker \cite{Frenod-Sonnen_Homogenization} giving the Two-Scale convergence of $f^{\epsilon}$. Roughly speaking they claim that, when $\epsilon$ is small,
\begin{eqnarray}
&& f^{\epsilon}(t,\mathbf{x},\mathbf{v}) \textrm{ is close to } G\big(t,\mathbf{x},\ucar(
\frac{t}{\epsilon},\mathbf{v})\big) \textrm{ in a strong sense, } \label{1nd-conv}\\
&& f^{\epsilon}(t,\mathbf{x},\mathbf{v}) \textrm{ is close to } f(t,\mathbf{x},\mathbf{v})
\textrm{ in a weak sense, }\label{2nd-conv}
\end{eqnarray}
and the equations for $G$ and $f$ are known ($\ucar(\tau,\mathbf{v})$ is defined by \eqref{DefUcar}). Then we show that the $\mathbf{v}$-dependence of $f$ is only through $v_{||}=\mathbf{v}\cdot\mathbf{e}_{1}$, $v_{\perp}=\sqrt{(\mathbf{v}\cdot\mathbf{e}_{2})^2+(\mathbf{v}\cdot\mathbf{e}_{3})^2}$ and not through the angular velocity component $\alpha$, which is such that $\mathbf{v}\cdot\mathbf{e}_{2}=v_{\perp}\cos\alpha$. Based on this fact, in Section \ref{SecUMMD}, we build a classical Macro-Micro decomposition of (\ref{Vlasov_rescaled_intro}) by following the same ideas as those of Lemou \& Mieussens \cite{Lemou-Mieussens}. It consists in decomposing $f^{\epsilon}$ in the following way 
\begin{equation}\label{eq:decolimf}
f^{\epsilon}\big(t,\mathbf{x},v_{\parallel}\;\mathbf{e_1}+v_{\perp}(\cos\alpha\;\mathbf{e_2} +
\sin\alpha\;\mathbf{e_3})\big) = \widetilde{m^{\epsilon}}(t,\mathbf{x},v_{||},v_{\perp}) + 
\widetilde{n^{\epsilon}}(t,\mathbf{x},v_{||},v_{\perp},\alpha)
\end{equation}
for some functions $\widetilde{m^{\epsilon}}$ and $\widetilde{n^{\epsilon}}$. Indeed, we will see in details that this decomposition is unique for every $\epsilon$,
small or not. Note that $\widetilde{m^{\epsilon}}$ does not depend on $\alpha$, as $f\big(t,\mathbf{x},v_{\parallel}\;\mathbf{e_1}+v_{\perp}(\cos\alpha\;\mathbf{e_2} + \sin\alpha\;\mathbf{e_3})\big)$, while $\widetilde{n^{\epsilon}}$ does. Equations for $\widetilde{m^{\epsilon}}$ and $\widetilde{n^{\epsilon}}$ are set out, making up the Macro-Micro model. The drawback of this model is that, because of \eqref{2nd-conv}, the exact equality \eqref{eq:decolimf} implies that function $\widetilde{n^{\epsilon}}$ is still highly oscillating, as $f^{\epsilon}$. To avoid this drawback, we develop the Two-Scale Macro-Micro framework. This is done in Sections \ref{SecTSMMDP}, \ref{SecTMaE}, and \ref{SecTMiE}, where we build the Two-Scale Macro-Micro approximation of $f^{\epsilon}$. First, based on results of Fr\'enod, Raviart \& Sonnendr\"ucker \cite{Raviart} we improve \eqref{1nd-conv} by using a first order approximation, claiming that when $\epsilon$ is small
\begin{equation}\label{3rd-conv}
f^{\epsilon}(t,\mathbf{x},\mathbf{v}) \,\textrm{ is close to }\, G\big(t,\mathbf{x},\ucar(
\frac{t}{\epsilon},\mathbf{v})\big) +\,\epsilon\,G_1\big(t,\mathbf{x},\ucar(
\frac{t}{\epsilon},\mathbf{v})\big) +\,\epsilon\,l\big(t,\frac{t}{\epsilon},\mathbf{x},
\mathbf{v})\big)\, \textrm{ in a strong sense, }
\end{equation}
where $G_1$ is the solution of a partial differential equation and $l$ is given by a formula (see \cite{Raviart}). Inspired by \eqref{3rd-conv}, we construct our Two-Scale Macro-Micro decomposition in the following way
\begin{equation}\label{TSMMdeco}
f^{\epsilon}(t,\mathbf{x},\mathbf{v})
= G\big(t,\mathbf{x},\ucar(\frac{t}{\epsilon},
\mathbf{v})\big)\,+\, \epsilon \,G_{1}^{\epsilon}\big(t,\mathbf{x},\ucar(\cfrac{t}{\epsilon},
\mathbf{v})\big)
+ \epsilon \, l(t,\cfrac{t}{\epsilon},\mathbf{x},\mathbf{v}) + \epsilon \, h^{\epsilon}
(t,\cfrac{t}{\epsilon},\mathbf{x},\mathbf{v}),
\end{equation}
where $G_{1}^{\epsilon}$ is intended to be close to $G_1$ when $\epsilon$ is small and $h^{\epsilon}$ is the corrector to be taken into account when the order of magnitude of $\epsilon$ is $1$. In other words, we use the Macro-Micro decomposition in order to make precise for every $\epsilon$ the approximation in \eqref{3rd-conv} which is valid only for small $\epsilon$. Then, we need to show that the decomposition in \eqref{TSMMdeco} is unique. This is done by paying attention to the properties of the operator 
\begin{gather}\label{the-oper}
 \cfrac{\D}{\D\tau} + (\mathbf{v} \times \mathcal{M}) \cdot \Nabla_{\mathbf{v}},
\end{gather}
since the expression \eqref{TSMMdeco} may also be set in the following form:
\begin{equation} \label{feps_decompo_intro}
f^{\epsilon}(t,\mathbf{x},\mathbf{v}) = \mathcal{A}^{\epsilon}(t,\frac{t}{\epsilon},
\mathbf{x},\mathbf{v})+\mathcal{B}^{\epsilon}(t,\cfrac{t}{\epsilon},\mathbf{x},
\mathbf{v}) \, ,
\end{equation}
with
\begin{equation} \label{decompo_intro_2}
\mathcal{A}^{\epsilon} = \mathcal{A}^{\epsilon}(t,\tau,\mathbf{x},\mathbf{v}) \in \textnormal{Ker}
\big( \cfrac{\D}{\D\tau} + (\mathbf{v} \times \mathcal{M}) \cdot \Nabla_{\mathbf{v}}\big)
\;\;\,\text{and}\;\;\,
\mathcal{B}^{\epsilon}= \mathcal{B}^{\epsilon}(t,\tau,\mathbf{x},\mathbf{v}) \in \Big(
\textnormal{Ker}\big( \cfrac{\D}{\D\tau} + (\mathbf{v} \times \mathcal{M}) \cdot
\Nabla_{\mathbf{v}}\big)\Big)^{\perp}\!\!\!.
\end{equation}
Then, we obtain the equations for the Macro part $G_{1}^{\epsilon}$ and the Micro one $h^{\epsilon}$. In Section \ref{SecRes}, we study the asymptotic behaviour of the Two-Scale Macro-Micro decomposed model as $\epsilon \to 0$. We close the paper with two appendices about some properties of the operator in \eqref{the-oper} leading mainly to the uniqueness of the decomposition in \eqref{feps_decompo_intro}.

\section{Two-Scale convergence of $f^{\epsilon}$} 
\label{SecTSC} 
\setcounter{equation}{0}
In this section, we briefly recall the results of Fr\'enod \& Sonnendr\"ucker
\cite{Frenod-Sonnen_Homogenization} where the asymptotic behaviour of
\eqref{Vlasov_rescaled_intro} is studied. First, under the following hypotheses
\begin{equation} \label{hypotheses_f0}
f_{0} \geq 0 \, , \qquad f_{0} \in L^{2}(\R^{6}) \, ,
\end{equation}
and
\begin{equation} \label{hypotheses_EB}
\begin{split}
\mathbf{E}^{\epsilon} \to \mathbf{E} \quad \textnormal{strong in
$L^{\infty}\big(0,T;L_{\textrm{loc}}^{2}(\R^{3})\big)$,} \\
\mathbf{B}^{\epsilon} \to \mathbf{B} \quad \textnormal{strong in
$L^{\infty}\big(0,T;L_{\textrm{loc}}^{2}(\R^{3})\big)$,}
\end{split}
\end{equation}
we have that, for any $T > 0$, $(f^{\epsilon})_{\epsilon\,>\,0}$ is bounded in
$L^{\infty}\big(0,T;L^{2}(\R^{6})\big)$. Second, as $\epsilon \to 0$, up to a subsequence,
$(f^{\epsilon})_{\epsilon\,>\,0}$ Two-Scale converges towards (see \cite{Allaire,Nguetseng})
\begin{equation}
F = F(t,\tau,\mathbf{x},\mathbf{v}) \in L^{\infty}\big(0,T;L_{\#_{1}}^{\infty}(\R^{+};L^{2}(\R^{6}))\big),
\end{equation}
where the subscript $\#_{1}$ indicates that $\tau \mapsto F(t,\tau,\mathbf{x},\mathbf{v})$ is $1$-periodic for any $(t,\mathbf{x},\mathbf{v})$. Introduced in \cite{Nguetseng} and further analyzed and used in several homogenization problems in \cite{Allaire}, the Two-Scale convergence means that for any regular function $\psi = \psi(t,\tau,\mathbf{x},\mathbf{v})$ such that $(t,\mathbf{x},\mathbf{v})\mapsto \psi(t,\tau,\mathbf{x},\mathbf{v})$ is compactly supported in $\Omega = [0,T) \times \R^{6}$ for any $\tau$, and $\tau \mapsto \psi(t,\tau,\mathbf{x},\mathbf{v})$ is $1$-periodic for any $(t,\mathbf{x},\mathbf{v})$, we have
\begin{equation}
\lim_{\epsilon\,\to\,0} \int_{\Omega} f^{\epsilon}\,(\psi)^{\epsilon}\,dt\,d\mathbf{x}\,d\mathbf{v} = \int_{\Omega}\int_{0}^{1} F\,\psi\,d\tau \, dt\,d\mathbf{x}\,d\mathbf{v} \, ,
\end{equation}
where
\begin{equation} \label{def_psieps}
(\psi)^{\epsilon} = (\psi)^{\epsilon}(t,\mathbf{x},\mathbf{v}) = \psi(t,\frac{t}{\epsilon},\mathbf{x},\mathbf{v}) \, .
\end{equation}
It was shown in \cite{Frenod-Sonnen_Homogenization} that the Two-Scale limit $F$ has the following
property:
\begin{equation} \label{F_in_Ker}
F \in \textnormal{Ker}\big( \cfrac{\D}{\D\tau} + (\mathbf{v} \times \mathcal{M}) \cdot \Nabla_{\mathbf{v}}\big) \, ,
\end{equation}
where $\frac{\D}{\D\tau} + (\mathbf{v} \times \mathcal{M}) \cdot \Nabla_{\mathbf{v}}$ is seen as anti-symmetric, non-bounded, and with open range from $L^{\infty}\big(0,T;L_{\#_{1}}^{\infty}(\R^{+};$ $ L^{2}(\R^{6}))\big)$ on itself. The direct consequence of (\ref{F_in_Ker}) is that, introducing the rotation matrix $r$ as
\begin{equation} \label{mat_rotation}
r(\tau) = \left(
\begin{array}{ccc}
1 & 0 & 0 \\
0 & \cos\,(2\pi\tau) & -\sin\,(2\pi\tau) \\
0 & \sin\,(2\pi\tau) & \cos\,(2\pi\tau)
\end{array}
\right) \, ,
\end{equation}
the Two-Scale limit $F$ expresses itself in terms of a function $G = G(t,\mathbf{x},\mathbf{u}) \in L^{\infty}\big(0,T;L^{2}(\R^{3})\big)$ by
\begin{equation} \label{link_FG}
F(t,\tau,\mathbf{x},\mathbf{v}) = G\big(t,\mathbf{x},\ucar(\tau,\mathbf{v})\big) \, ,
\end{equation}
where
\begin{equation}
\label{DefUcar} 
\ucar(\tau,\mathbf{v}) = r(\tau)\,\mathbf{v} \, .
\end{equation}
Using oscillating test functions defined by (\ref{def_psieps}) for regular functions $\psi$ which are moreover in $\textnormal{Ker}\big( \frac{\D}{\D\tau} + (\mathbf{v} \times \mathcal{M}) \cdot \Nabla_{\mathbf{v}}\big)$, it may be brought out that $G$ is the solution of
\begin{equation} \label{eq_G}
\left\{
\begin{array}{l}
\cfrac{\D G}{\D t} + \mathbf{u}_{||} \cdot \Nabla_{\mathbf{x}} G + (\mathbf{E}_{||} + \mathbf{u} \times \mathbf{B}_{||}) \cdot \Nabla_{\mathbf{u}} G = 0 \, , \\
G(t=0,\mathbf{x},\mathbf{u}) = f_{0}(\mathbf{x},\mathbf{u}) \, ,
\end{array}
\right.
\end{equation}
where we have used the notation $\mathbf{u}_{||} = (\mathbf{u}\cdot\mathbf{e}_{1})\,\mathbf{e}_{1}$. \\~\\
\indent Having in mind diagram (\ref{diagram_TSAPS}), we can say at this point that equation (\ref{Vlasov_rescaled_intro}) plays the role of $E^{\epsilon}\,u^{\epsilon} = 0$, and (\ref{link_FG})-(\ref{eq_G}) plays the role of $\mathcal{E}\,U = 0$. So, if we build a numerical method to approximate (\ref{eq_G}), this numerical method, coupled with (\ref{link_FG}) would play the role of $\mathcal{E}_{\Delta z}\,U_{\Delta z} = 0$. \\

\indent In \cite{Frenod-Sonnen_Homogenization}, it was also proved that $f$, the weak-$*$
limit of $(f^{\epsilon})_{\epsilon>0}$, is the solution to
\begin{equation} \label{eq_f_weak}
\left\{
\begin{array}{l}
\cfrac{\D f}{\D t} + \mathbf{v}_{||} \cdot \Nabla_{\mathbf{x}} f+ (\mathbf{E}_{||} + \mathbf{v} \times \mathbf{B}_{||}) \cdot \Nabla_{\mathbf{v}}f = 0 \, , \\
\displaystyle f(t=0,\mathbf{x},\mathbf{v}) = \int_{0}^{1} f_{0}\big(\mathbf{x},\ucar(\tau,\mathbf{v})\big) \, d\tau \, .
\end{array}
\right.
\end{equation}
This equation is obtained from (\ref{link_FG})-(\ref{eq_G}) by performing an integration in $\tau$ and by using the link between the weak-$*$ limit $f$ and the Two-Scale limit $F$ given by
\begin{equation} \label{link_fF}
f(t,\mathbf{x},\mathbf{v}) = \int_{0}^{1} F(t,\tau,\mathbf{x},\mathbf{v})\,d\tau \, .
\end{equation}
Then, equation (\ref{eq_f_weak}) plays the role of $E\,u=0$ in diagram (\ref{diagram_TSAPS}) and any of its numerical
approximations would play the role of $E_{\Delta z}\,u_{\Delta z} = 0$.

\section{Classical Macro-Micro decomposition}
\label{SecUMMD}
\setcounter{equation}{0}

\subsection{Rewriting the weak-$*$ limit model}
Before going further, we will rewrite model \eqref{eq_f_weak} in a form involving cylindrical
variables in velocity. In this coordinate system, the weak-$*$ limit is independent of the
angle variable, and consequently it is easily gotten that the $\mathbf{v}\times\mathbf{B}_{||}$
piece in the force term in \eqref{eq_f_weak} is removed.\\

Because of the average over $\tau$ and of the definition of $\ucar(\tau,\mathbf{v})$, it is obvious to see that the initial data in (\ref{eq_f_weak}) only depends on $\mathbf{x}$, $\mathbf{v}_{||}$ and $v_{\perp} = \sqrt{v_{2}^{2}+v_{3}^{2}}$. In other words, writing the velocity in the following cylindrical variables
\begin{equation}\label{cyl:coord}
\left\{
\begin{array}{l}
\textnormal{$v_{||}\in\R$ such that $v_{||} = \mathbf{v} \cdot \mathbf{e}_{1} \,$,} \\
\textnormal{$v_{\perp}\in\R^+$ such that $v_{\perp} = \sqrt{v_{2}^{2}+v_{3}^{2}} \, $,} \\
\textnormal{$\alpha\in [0,2\pi]$ such that $v_{2} = v_{\perp}\,\cos\alpha$ and $v_{3} = v_{\perp}\,\sin\alpha$,}
\end{array}
\right.
\end{equation}
there exists a function $m_{0} = m_{0}(\mathbf{x},v_{||},v_{\perp})$, independent of $\alpha$, such that
\begin{equation}\label{eq:m_0}
m_{0}(\mathbf{x},v_{||},v_{\perp}) = \cfrac{1}{2\pi} \, \int_{0}^{2\pi} f_{0}\big(\mathbf{x},\ucar(\tau,v_{||}\,\mathbf{e}_{1} + v_{\perp}\,\cos\alpha\,\mathbf{e}_{2} + v_{\perp}\,\sin\alpha\,\mathbf{e}_{3})\big)\,d\alpha \,.
\end{equation}
The quickest way to get this is to notice that
\begin{equation}
\ucar(\tau,v_{||}\,\mathbf{e}_{1} + v_{\perp}\,\cos\alpha\,\mathbf{e}_{2} + v_{\perp}\,\sin\alpha\,\mathbf{e}_{3}) = \left(
\begin{array}{c}
v_{||} \\ v_{\perp}\,\cos(\alpha+2\pi\tau) \\ v_{\perp}\,\sin(\alpha+2\pi\tau)
\end{array}
\right) \, ,
\end{equation}
and that, consequently, integrating in $\tau$ suppresses the $\alpha$-dependency. Hence,
\eqref{eq:m_0} writes also 
\begin{equation}
m_{0}(\mathbf{x},v_{||},v_{\perp}) = \cfrac{1}{2\pi} \, \int_{0}^{2\pi}f_{0}(\mathbf{x},v_{||}\,\mathbf{e}_{1} + v_{\perp}\,\cos\alpha\,\mathbf{e}_{2} + v_{\perp}\,\sin\alpha\,\mathbf{e}_{3})\,d\alpha \,.
\end{equation}
\indent Concerning the function $f$ itself, we introduce the function $m = m(t,\mathbf{x},v_{||},v_{\perp},\alpha)$ linked with $f$ by
\begin{equation}\label{link:mf}
m(t,\mathbf{x},v_{||},v_{\perp},\alpha) = f(t,\mathbf{x},v_{||}\,\mathbf{e}_{1}+ v_{\perp}\,\cos\alpha\,\mathbf{e}_{2} + v_{\perp}\,\sin\alpha\,\mathbf{e}_{3}) \, .
\end{equation}
Because of (\ref{link_fF}) and (\ref{link_FG}), and using the same argument as just above, it is easily seen that $m$ does not depend on $\alpha$ and that
\begin{equation}
\begin{split}
m(t,\mathbf{x},v_{||},v_{\perp}) &= \int_{0}^{1}G(t,\mathbf{x},v_{||}\,\mathbf{e}_{1}+ v_{\perp}\,\cos(\alpha+2\pi\tau)\,\mathbf{e}_{2} + v_{\perp}\,\sin(\alpha+2\pi\tau)\,\mathbf{e}_{3}) \, d\tau \\
&= \frac{1}{2\pi}\int_{0}^{2\pi}G(t,\mathbf{x},v_{||}\,\mathbf{e}_{1}+ v_{\perp}\,\cos\alpha\,\mathbf{e}_{2} + v_{\perp}\,\sin\alpha\,\mathbf{e}_{3}) \, d\alpha \, .
\end{split}
\end{equation}
Noticing that
\begin{equation}
\begin{split}
\cfrac{\D m}{\D v_{\perp}}(t,\mathbf{x},v_{||},v_{\perp}) &= \cos\alpha\,\cfrac{\D f}{\D v_{2}} (t,\mathbf{x},v_{||}\,\mathbf{e}_{1}+ v_{\perp}\,\cos\alpha\,\mathbf{e}_{2} + v_{\perp}\,\sin\alpha\,\mathbf{e}_{3}) \\
&\qquad + \sin\alpha\,\cfrac{\D f}{\D v_{3}} (t,\mathbf{x},v_{||}\,\mathbf{e}_{1}+ v_{\perp}\,\cos\alpha\,\mathbf{e}_{2} + v_{\perp}\,\sin\alpha\,\mathbf{e}_{3}) \, , \\
\cfrac{\D m}{\D\alpha}(t,\mathbf{x},v_{||},v_{\perp}) &= -v_{\perp}\,\sin\alpha\,\cfrac{\D f}{\D v_{2}} (t,\mathbf{x},v_{||}\,\mathbf{e}_{1}+ v_{\perp}\,\cos\alpha\,\mathbf{e}_{2} + v_{\perp}\,\sin\alpha\,\mathbf{e}_{3}) \\
&\qquad + v_{\perp}\,\cos\alpha\,\cfrac{\D f}{\D v_{3}} (t,\mathbf{x},v_{||}\,\mathbf{e}_{1}+ v_{\perp}\,\cos\alpha\,\mathbf{e}_{2} + v_{\perp}\,\sin\alpha\,\mathbf{e}_{3}) \, ,
\end{split}
\end{equation}
the equation for $m$ can be deduced from (\ref{eq_f_weak}). For this purpose, we develop 
$\mathbf{v} \times \mathbf{B}_{||}$ in the new variables:
\begin{equation}
v_{\perp}(\cos\alpha\,\mathbf{e}_{2} + \sin\alpha\,\mathbf{e}_{3}) \times \mathbf{B}_{||} = |B_{||}| \, v_{\perp}\, (\sin\alpha\,\mathbf{e}_{2} - \cos\alpha\, \mathbf{e}_{3}) \, .
\end{equation}
As a consequence, we have
\begin{equation}
\begin{split}
\big(v_{\perp}(\cos\alpha\,\mathbf{e}_{2} + \sin\alpha\,\mathbf{e}_{3}) \times \mathbf{B}_{||} \big) \cdot \Nabla_{\mathbf{v}} f(t,\mathbf{x},v_{||}\,\mathbf{e}_{1}+ v_{\perp}\,\cos\alpha\,\mathbf{e}_{2} &+ v_{\perp}\,\sin\alpha\,\mathbf{e}_{3}) \\
&= -|B_{||}|\,\cfrac{\D m}{\D \alpha}(t,\mathbf{x},v_{||},v_{\perp}) = 0 \, ,
\end{split}
\end{equation}
meaning that the force term in (\ref{eq_f_weak}) may be reduced to $\mathbf{E}_{||}$. Then, $m$ is the solution of
\begin{equation} \label{eq_m}
\left\{
\begin{array}{l}
\cfrac{\D m}{\D t} + v_{||} \, \cfrac{\D m}{\D x_{||}} + E_{||}\,\cfrac{\D m}{\D v_{||}} = 0 \, , \\
m(t=0,\mathbf{x},v_{||},v_{\perp}) = m_{0}(\mathbf{x},v_{||},v_{\perp}) \, .
\end{array}
\right.
\end{equation}

\subsection{Macro-Micro decomposition}
\label{Sec:Mmdeco}
Inspired by \eqref{eq_f_weak} or more precisely by \eqref{eq_m}, we now write a classical
Macro-Micro decomposition of the solution $f^{\epsilon}$ of \eqref{Vlasov_rescaled_intro}.
The aim of this decomposition is (see \cite{Bennoune, Lemou-Mieussens}) to build a new
model which is an equivalent reformulation of \eqref{Vlasov_rescaled_intro} for all
$\epsilon > 0$ and which straightforwardly gives the weak-$*$ limit model \eqref{eq_m}
when $\epsilon\to0$. Thus, this new formulation of problem \eqref{Vlasov_rescaled_intro}
by using the Macro-Micro decomposition is expected to be well adapted to the building
of an Asymptotic-Preserving scheme.

The main idea of the Macro-Micro decomposition is to write $f^{\epsilon}$ as a sum of a function
being in $\textnormal{Ker}\big((\mathbf{v} \times \mathcal{M})\cdot \Nabla_{\mathbf{v}}\big)$ and
a function being in $\big(\textnormal{Ker}\big((\mathbf{v} \times \mathcal{M})\cdot
\Nabla_{\mathbf{v}}\big)\big)^{\perp}$. Adapting this idea to our context and having in mind
\eqref{eq_m}, we seek $m_{1}^{\epsilon} = m_{1}^{\epsilon}(t,\mathbf{x},v_{||},v_{\perp}) \in
L^{\infty}\big(0,T;L^{2}(\R^{3} \times \R \times \R^{+}; v_{\perp}\,d\mathbf{x}\,dv_{||}\,dv_{\perp})
\big)$ and $p^{\epsilon} = p^{\epsilon}(t,\mathbf{x},v_{||},v_{\perp},\alpha) \in L^{\infty}
\big(0,T;L_{\#_{2\pi}}^{2}(\R;L^{2}(\R^{3} \times \R \times \R^{+}; v_{\perp}\,d\mathbf{x}\,dv_{||}\,
dv_{\perp}))\big)$ such that
\begin{equation} \label{def_neps}
n^{\epsilon}(t,\mathbf{x},v_{||},v_{\perp},\alpha) = \cfrac{\D p^{\epsilon}}{\D\alpha}(t,\mathbf{x},v_{||},v_{\perp},\alpha) \, ,
\end{equation}
and
\begin{equation} \label{decompo_feps_weak}
f^{\epsilon}(t,\mathbf{x},v_{||}\,\mathbf{e}_{1}+v_{\perp}\,\cos\alpha\,\mathbf{e}_{2}+v_{\perp}\,\sin\alpha\,\mathbf{e}_{3}) = m(t,\mathbf{x},v_{||},v_{\perp}) + m_{1}^{\epsilon}(t,\mathbf{x},v_{||},v_{\perp}) + n^{\epsilon}(t,\mathbf{x},v_{||},v_{\perp},\alpha).
\end{equation}
The subscript $\#_{2\pi}$ indicates that $\alpha\mapsto p^{\epsilon}(t,\mathbf{x},v_{||},
v_{\perp},\alpha)$ is $2\pi$-periodic for any $(t,\mathbf{x},v_{||},v_{\perp})$.

\indent We have proved in the previous section that $m \in \textnormal{Ker}\big(
\cfrac{\D}{\D\alpha}\big)$. In addition, considering the anti-symmetric, non-bounded
operator with closed range\footnote{See Appendix A.}
\begin{equation}
\begin{split}
\cfrac{\D}{\D\alpha} \, : \, L^{\infty}\big(0,T;L_{\#_{2\pi}}^{2}(\R;L^2(\R^{3} &\times \R \times \R^{+}; v_{\perp}\,d\mathbf{x}\,dv_{||}\,dv_{\perp}))\big) \\
&\longrightarrow \, L^{\infty}\big(0,T;L_{\#_{2\pi}}^{2}(\R;L^2(\R^{3} \times \R \times \R^{+}; v_{\perp}\,d\mathbf{x}\,dv_{||}\,dv_{\perp}))\big) \, ,
\end{split}
\end{equation}
we clearly have that
\begin{equation}
m_{1}^{\epsilon} \in \textnormal{Ker}\big(\cfrac{\D}{\D\alpha}\big) \qquad
\textnormal{and} \qquad n^{\epsilon} \in \Big( \textnormal{Ker}\big(\cfrac{\D}{\D\alpha}
\big)\Big)^{\perp} = \textnormal{Im}\big(\cfrac{\D}{\D\alpha}\big)
\end{equation}
exist and are unique.
%
Then, if we inject the decomposition (\ref{decompo_feps_weak}) in Vlasov equation (\ref{Vlasov_rescaled_intro}), we obtain
\begin{equation} \label{Vlasov_with_mm1n}
\begin{split}
\cfrac{\D m}{\D t} &+ \cfrac{\D m_{1}^{\epsilon}}{\D t} + \cfrac{\D n^{\epsilon}}{\D t} + \left(
\begin{array}{c}
v_{||} \\ v_{\perp}\,\cos\alpha \\ v_{\perp}\,\sin\alpha
\end{array}
\right) \cdot \Nabla_{\mathbf{x}} (m+m_{1}^{\epsilon}+n^{\epsilon}) \\
&+ \left(
\begin{array}{c}
E_{||}^{\epsilon} + v_{\perp}\,(\cos\alpha \, B_{3}^{\epsilon}-\sin\alpha\,B_{2}^{\epsilon}) \\
\cos\alpha \, E_{2}^{\epsilon} + \sin\alpha \,E_{3}^{\epsilon} - v_{||}\,(\cos\alpha \, B_{3}^{\epsilon}-\sin\alpha\,B_{2}^{\epsilon}) \\
\cfrac{1}{v_{\perp}}\, \big(-\sin\alpha \, E_{2}^{\epsilon} + \cos\alpha \,E_{3}^{\epsilon} + v_{||}\,(\sin\alpha \, B_{3}^{\epsilon}+\cos\alpha\,B_{2}^{\epsilon})\big)-B_{||}^{\epsilon}-\cfrac{1}{\epsilon}
\end{array}
\right) \\
&\qquad \qquad \qquad \qquad \qquad \qquad \qquad \qquad \qquad \qquad \cdot \left(
\begin{array}{c}
\cfrac{\D m}{\D v_{||}} + \cfrac{\D m_{1}^{\epsilon}}{\D v_{||}} + \cfrac{\D n^{\epsilon}}{\D v_{||}} \\
\cfrac{\D m}{\D v_{\perp}} + \cfrac{\D m_{1}^{\epsilon}}{\D v_{\perp}} + \cfrac{\D n^{\epsilon}}{\D v_{\perp}} \\
\cfrac{\D n^{\epsilon}}{\D \alpha}
\end{array}
\right) = 0 \, .
\end{split}
\end{equation}
Combining this equation with (\ref{eq_m}), we finally obtain
\begin{equation}
\begin{split}
&\cfrac{\D m_{1}^{\epsilon}}{\D t} + \cfrac{\D n^{\epsilon}}{\D t} + \left(
\begin{array}{c}
0 \\ v_{\perp}\,\cos\alpha \\ v_{\perp}\,\sin\alpha
\end{array}
\right) \cdot \Nabla_{\mathbf{x}} m + \left(
\begin{array}{c}
v_{||} \\ v_{\perp}\,\cos\alpha \\ v_{\perp}\,\sin\alpha
\end{array}
\right) \cdot \Nabla_{\mathbf{x}} (m_{1}^{\epsilon}+n^{\epsilon}) \\
&+ \left(
\begin{array}{c}
E_{||}^{\epsilon}-E_{||} + v_{\perp}\,(\cos\alpha \, B_{3}^{\epsilon}-\sin\alpha\,B_{2}^{\epsilon}) \\
\cos\alpha \, E_{2}^{\epsilon} + \sin\alpha \,E_{3}^{\epsilon} - v_{||}\,(\cos\alpha \, B_{3}^{\epsilon}-\sin\alpha\,B_{2}^{\epsilon}) \\
\cfrac{1}{v_{\perp}}\, \big(-\sin\alpha \, E_{2}^{\epsilon} + \cos\alpha \,E_{3}^{\epsilon} + v_{||}\,(\sin\alpha \, B_{3}^{\epsilon}+\cos\alpha\,B_{2}^{\epsilon})\big)-B_{||}^{\epsilon}-\cfrac{1}{\epsilon}
\end{array}
\right) \cdot \left(
\begin{array}{c}
\cfrac{\D m}{\D v_{||}} \\
\cfrac{\D m}{\D v_{\perp}} \\
0
\end{array}
\right) \\
&+ \left(
\begin{array}{c}
E_{||}^{\epsilon} + v_{\perp}\,(\cos\alpha \, B_{3}^{\epsilon}-\sin\alpha\,B_{2}^{\epsilon}) \\
\cos\alpha \, E_{2}^{\epsilon} + \sin\alpha \,E_{3}^{\epsilon} - v_{||}\,(\cos\alpha \, B_{3}^{\epsilon}-\sin\alpha\,B_{2}^{\epsilon}) \\
\cfrac{1}{v_{\perp}}\, \big(-\sin\alpha \, E_{2}^{\epsilon} + \cos\alpha \,E_{3}^{\epsilon} + v_{||}\,(\sin\alpha \, B_{3}^{\epsilon}+\cos\alpha\,B_{2}^{\epsilon})\big)-B_{||}^{\epsilon}-\cfrac{1}{\epsilon}
\end{array}
\right) \cdot \left(
\begin{array}{c}
\cfrac{\D m_{1}^{\epsilon}}{\D v_{||}} + \cfrac{\D n^{\epsilon}}{\D v_{||}} \\
\cfrac{\D m_{1}^{\epsilon}}{\D v_{\perp}} + \cfrac{\D n^{\epsilon}}{\D v_{\perp}} \\
\cfrac{\D n^{\epsilon}}{\D \alpha}
\end{array}
\right) = 0 \, .
\end{split}
\end{equation}

\indent By projecting (\ref{Vlasov_with_mm1n}) onto $\textnormal{Ker}(\frac{\D}{\D\alpha})$ 
and $\big(\textnormal{Ker}(\frac{\D}{\D\alpha})\big)^{\perp} = \textnormal{Im}(\frac{\D}{\D\alpha})$, 
we will obtain the Macro-Micro decomposition of $f^{\epsilon}$. In order to do this, we use the fact
that (see Appendix A and \cite{Bostan_2010}) projecting a function onto 
$\textnormal{Ker}(\frac{\D}{\D\alpha})$ consists in computing its average in $\alpha$ 
and projecting it onto $\textnormal{Im}(\frac{\D}{\D\alpha})$ consists in substracting 
from it its average value. This is what we do in the next lines. \\
\indent Since $m$, $m_{1}^{\epsilon}$, $\mathbf{E}$, $\mathbf{E}^{\epsilon}$, $\mathbf{B}$ 
and $\mathbf{B}^{\epsilon}$ do not depend on $\alpha$, and recalling the definition of 
$n^{\epsilon}$ given in (\ref{def_neps}), the projection of (\ref{Vlasov_with_mm1n}) 
onto $\textnormal{Ker}(\frac{\D}{\D\alpha})$ gives
\begin{equation} \label{proj_Ker_weak}
\begin{split}
\cfrac{\D m_{1}^{\epsilon}}{\D t} &+ v_{||}\,\cfrac{\D m_{1}^{\epsilon}}{\D x_{||}} + E_{||}^{\epsilon}\, \cfrac{\D m_{1}^{\epsilon}}{\D v_{||}} \\
&= -(E_{||}^{\epsilon}-E_{||})\,\cfrac{\D m}{\D v_{||}} - \cfrac{v_{\perp}}{2\pi} \, \left(
\begin{array}{c}
0 \\ 1 \\ 1
\end{array}
\right) \cdot \left(
\begin{array}{c}
0 \\
\displaystyle \int_{0}^{2\pi} \sin\alpha\,\cfrac{\D p^{\epsilon}}{\D x_{2}} \, d\alpha \\
-\displaystyle \int_{0}^{2\pi} \cos\alpha\,\cfrac{\D p^{\epsilon}}{\D x_{3}} \, d\alpha
\end{array}
\right) \\
&\qquad - \cfrac{1}{2\pi}\, \mathbf{E}^{\epsilon} \cdot \left(
\begin{array}{c}
0 \\
\displaystyle \int_{0}^{2\pi} \sin\alpha\,\big(\cfrac{\D p^{\epsilon}}{\D v_{\perp}}+\cfrac{p^{\epsilon}}{v_{\perp}}\big) \, d\alpha \\
-\displaystyle \int_{0}^{2\pi} \cos\alpha\,\big(\cfrac{\D p^{\epsilon}}{\D v_{\perp}}+\cfrac{p^{\epsilon}}{v_{\perp}}\big) \, d\alpha
\end{array}
\right) \\
&\qquad + \cfrac{1}{2\pi}\,\left(
\begin{array}{c}
-v_{\perp} \\ v_{\parallel}\,B_{3}^{\epsilon} \\ v_{\parallel}\,B_{2}^{\epsilon}
\end{array}
\right) \cdot \left(
\begin{array}{c}
\displaystyle \int_{0}^{2\pi} (\cos\alpha\,B_{2}^{\epsilon}+\sin\alpha\,B_{3}^{\epsilon})\,\cfrac{\D p^{\epsilon}}{\D v_{||}} \, d\alpha \\
\displaystyle \int_{0}^{2\pi} \sin\alpha\,\big(\cfrac{\D p^{\epsilon}}{\D v_{\perp}}+\cfrac{p^{\epsilon}}{v_{\perp}}\big) \, d\alpha \\
\displaystyle \int_{0}^{2\pi} \cos\alpha\,\big(\cfrac{\D p^{\epsilon}}{\D v_{\perp}}+\cfrac{p^{\epsilon}}{v_{\perp}}\big) \, d\alpha
\end{array}
\right) \, .
\end{split}
\end{equation}
In the second, third and last terms of the right hand side of \eqref{proj_Ker_weak}, we use integrations by parts to tackle
$\alpha$-integrations.
Finally, we perform the projection onto $\textnormal{Im}(\frac{\D}{\D\alpha})$ by substracting (\ref{proj_Ker_weak}) from (\ref{Vlasov_with_mm1n}), we obtain a second order equation for $p^{\epsilon}$ in which $m$ and $m_{1}^{\epsilon}$ appear as right hand side members:
\begin{equation} \label{proj_Im_weak}
\begin{split}
&\cfrac{\D^{2}p^{\epsilon}}{\D t \, \D\alpha} + \left(
\begin{array}{c}
v_{||} \\ v_{\perp} \\ v_{\perp}
\end{array}
\right) \cdot \left(
\begin{array}{c}
\cfrac{\D^{2}p^{\epsilon}}{\D x_{||}\,\D\alpha} \\
\displaystyle \cos\alpha\,\cfrac{\D^{2}p^{\epsilon}}{\D x_{2}\,\D\alpha} - \cfrac{1}{2\pi}\,\int_{0}^{2\pi} \cos\alpha\,\cfrac{\D^{2}p^{\epsilon}}{\D x_{2}\,\D\alpha} \, d\alpha \\
\displaystyle \sin\alpha\,\cfrac{\D^{2}p^{\epsilon}}{\D x_{3}\,\D\alpha} - \cfrac{1}{2\pi}\,\int_{0}^{2\pi} \sin\alpha \,\cfrac{\D^{2}p^{\epsilon}}{\D x_{3}\,\D\alpha}\, d\alpha
\end{array}
\right) \\
& + \mathbf{E}^{\epsilon} \cdot \left(
\begin{array}{c}
\cfrac{\D^{2}p^{\epsilon}}{\D v_{||}\,\D\alpha} \\
\displaystyle \cos\alpha\,\cfrac{\D^{2}p^{\epsilon}}{\D v_{\perp}\,\D\alpha} - \cfrac{1}{2\pi}\,\int_{0}^{2\pi} \cos\alpha\,\cfrac{\D^{2}p^{\epsilon}}{\D v_{\perp}\,\D\alpha} \, d\alpha - \cfrac{1}{v_{\perp}}\,\sin\alpha\,\cfrac{\D^{2}p^{\epsilon}}{\D\alpha^{2}} + \cfrac{1}{2\pi\,v_{\perp}}\,\int_{0}^{2\pi} \sin\alpha \,\cfrac{\D^{2}p^{\epsilon}}{\D\alpha^{2}}\, d\alpha \\
\displaystyle \sin\alpha\,\cfrac{\D^{2}p^{\epsilon}}{\D v_{\perp}\,\D\alpha} - \cfrac{1}{2\pi}\,\int_{0}^{2\pi} \sin\alpha \,\cfrac{\D^{2}p^{\epsilon}}{\D v_{\perp}\,\D\alpha}\, d\alpha + \cfrac{1}{v_{\perp}}\,\cos\alpha\,\cfrac{\D^{2}p^{\epsilon}}{\D\alpha^{2}} - \cfrac{1}{2\pi\,v_{\perp}}\,\int_{0}^{2\pi} \cos\alpha \,\cfrac{\D^{2}p^{\epsilon}}{\D\alpha^{2}}\, d\alpha
\end{array}
\right) \\
&+ \left(
\begin{array}{c}
v_{\perp} \\ -v_{||}\,B_{3}^{\epsilon} \\ v_{||}\,B_{2}^{\epsilon}
\end{array}
\right) \cdot \left(
\begin{array}{c}
\displaystyle (\cos\alpha\,B_{3}^{\epsilon}-\sin\alpha\,B_{2}^{\epsilon})\,\cfrac{\D^{2}p^{\epsilon}}{\D v_{||}\,\D\alpha} - \cfrac{1}{2\pi}\,\int_{0}^{2\pi} (\cos\alpha\,B_{3}^{\epsilon}-\sin\alpha\,B_{2}^{\epsilon})\,\cfrac{\D^{2}p^{\epsilon}}{\D v_{||}\,\D\alpha} \, d\alpha \\
\displaystyle \cos\alpha \, \cfrac{\D^{2}p^{\epsilon}}{\D v_{\perp}\,\D\alpha} - \cfrac{1}{v_{\perp}} \, \sin\alpha \, \cfrac{\D^{2}p^{\epsilon}}{\D\alpha^{2}} - \cfrac{1}{2\pi} \, \int_{0}^{2\pi} \Big( \cos\alpha \, \cfrac{\D^{2}p^{\epsilon}}{\D v_{\perp}\,\D\alpha} - \cfrac{1}{v_{\perp}} \, \sin\alpha \, \cfrac{\D^{2}p^{\epsilon}}{\D\alpha^{2}} \Big) \, d\alpha \\
\displaystyle \sin\alpha \, \cfrac{\D^{2}p^{\epsilon}}{\D v_{\perp}\,\D\alpha} + \cfrac{1}{v_{\perp}} \, \cos\alpha \, \cfrac{\D^{2}p^{\epsilon}}{\D\alpha^{2}} - \cfrac{1}{2\pi} \, \int_{0}^{2\pi} \Big( \sin\alpha \, \cfrac{\D^{2}p^{\epsilon}}{\D v_{\perp}\,\D\alpha} + \cfrac{1}{v_{\perp}} \, \cos\alpha \, \cfrac{\D^{2}p^{\epsilon}}{\D\alpha^{2}} \Big) \, d\alpha 
\end{array}
\right) \\
&- (B_{||}^{\epsilon}+\cfrac{1}{\epsilon}) \, \cfrac{\D^{2}p^{\epsilon}}{\D\alpha^{2}} \\
&\qquad \qquad = - \left(
\begin{array}{c}
0 \\ v_{\perp}\,\cos\alpha \\ v_{\perp}\,\sin\alpha
\end{array}
\right) \cdot \Nabla_{\mathbf{x}}m - \mathbf{E}^{\epsilon} \cdot \left(
\begin{array}{c}
0 \\ \cos\alpha \, \cfrac{\D m}{\D v_{\perp}} \\ \sin\alpha \, \cfrac{\D m}{\D v_{\perp}}
\end{array}
\right) \\
&\qquad \qquad \qquad - \left(
\begin{array}{c}
v_{\perp}\,(\cos\alpha\,B_{3}^{\epsilon}-\sin\alpha\,B_{2}^{\epsilon}) \\ -v_{||}\,B_{3}^{\epsilon} \\ v_{||}\,B_{2}^{\epsilon}
\end{array}
\right) \cdot \left(
\begin{array}{c}
\cfrac{\D m}{\D v_{||}} \\ \cos\alpha \, \cfrac{\D m}{\D v_{\perp}} \\ \sin\alpha \, \cfrac{\D m}{\D v_{\perp}}
\end{array}
\right) \\
&\qquad \qquad \qquad - \left(
\begin{array}{c}
0 \\ v_{\perp}\,\cos\alpha \\ v_{\perp}\,\sin\alpha
\end{array}
\right) \cdot \Nabla_{\mathbf{x}}m_{1}^{\epsilon} - \mathbf{E}^{\epsilon} \cdot \left(
\begin{array}{c}
0 \\ \cos\alpha \, \cfrac{\D m_{1}^{\epsilon}}{\D v_{\perp}} \\ \sin\alpha \, \cfrac{\D m_{1}^{\epsilon}}{\D v_{\perp}}
\end{array}
\right) \\
&\qquad \qquad \qquad - \left(
\begin{array}{c}
v_{\perp}\,(\cos\alpha\,B_{3}^{\epsilon}-\sin\alpha\,B_{2}^{\epsilon}) \\ -v_{||}\,B_{3}^{\epsilon} \\ v_{||}\,B_{2}^{\epsilon}
\end{array}
\right) \cdot \left(
\begin{array}{c}
\cfrac{\D m_{1}^{\epsilon}}{\D v_{||}} \\ \cos\alpha \, \cfrac{\D m_{1}^{\epsilon}}{\D v_{\perp}} \\ \sin\alpha \, \cfrac{\D m_{1}^{\epsilon}}{\D v_{\perp}}
\end{array}
\right) \, .
\end{split}
\end{equation}

Concerning the initial conditions for $m_{1}^{\epsilon}$ and $p^{\epsilon}$, we invoke the initial distribution $f_{0}$ as follows:
\begin{equation}
\begin{split}
f_{0}(\mathbf{x},v_{||}\,\mathbf{e}_{1} + v_{\perp}\,\cos\alpha\,\mathbf{e}_{2} &+ v_{\perp}\,\sin\alpha\,\mathbf{e}_{3}) \\
&= \cfrac{1}{2\pi} \int_{0}^{2\pi} f_{0}(\mathbf{x},v_{||}\,\mathbf{e}_{1} + v_{\perp}\,\cos\theta\,\mathbf{e}_{2} + v_{\perp}\,\sin\theta\,\mathbf{e}_{3}) \, d\theta \\
&\qquad + \Bigg[ f_{0}(\mathbf{x},v_{||}\,\mathbf{e}_{1} + v_{\perp}\,\cos\alpha\,\mathbf{e}_{2} + v_{\perp}\,\sin\alpha\,\mathbf{e}_{3}) \\
&\qquad \qquad - \cfrac{1}{2\pi} \int_{0}^{2\pi} f_{0}(\mathbf{x},v_{||}\,\mathbf{e}_{1} + v_{\perp}\,\cos\theta\,\mathbf{e}_{2} + v_{\perp}\,\sin\theta\,\mathbf{e}_{3}) \, d\theta \Bigg] \\
&= m_{0}(\mathbf{x},v_{||},v_{\perp}) \\
&\qquad + \big[ f_{0}(\mathbf{x},v_{||}\,\mathbf{e}_{1} + v_{\perp}\,\cos\alpha\,\mathbf{e}_{2} + v_{\perp}\,\sin\alpha\,\mathbf{e}_{3}) - m_{0}(\mathbf{x},v_{||},v_{\perp})\big]. 
\end{split}
\end{equation}
Since $m_{|_{t\,=\,0}} = m_{0}$, we have
\begin{equation}
\begin{split}
m_{1}^{\epsilon}(t=0,\mathbf{x},v_{||},&v_{\perp}) + \cfrac{\D p_{\epsilon}}{\D\alpha}(t=0,\mathbf{x},v_{||},v_{\perp},\alpha) \\
&= f_{0}(\mathbf{x},v_{||}\,\mathbf{e}_{1} + v_{\perp}\,\cos\alpha\,\mathbf{e}_{2} + v_{\perp}\,\sin\alpha\,\mathbf{e}_{3}) - m_{0}(\mathbf{x},v_{||},v_{\perp}).
\end{split}
\end{equation}
By integrating this equation in $\alpha$ over $[0,2\pi]$, we find
\begin{equation} \label{proj_Ker_weak_init}
m_{1}^{\epsilon}(t=0,\mathbf{x},v_{||},v_{\perp}) = 0.
\end{equation}
Then we have
\begin{equation} \label{proj_Im_weak_init}
p^{\epsilon}(t=0,\mathbf{x},v_{||},v_{\perp},\alpha) = \int_{0}^{\alpha} f_{0}(\mathbf{x},v_{||}\,\mathbf{e}_{1} + v_{\perp}\,\cos\theta\,\mathbf{e}_{2} + v_{\perp}\,\sin\theta\,\mathbf{e}_{3}) \, d\theta - \alpha\,m_{0}(\mathbf{x},v_{||},v_{\perp}) \, .
\end{equation}
\begin{remark}\label{classical-Mm}
On the one hand, we mention that we have existence and uniqueness for solutions of the classical
Macro-Micro model \eqref{eq_m}, \eqref{proj_Ker_weak}, \eqref{proj_Im_weak}, \eqref{proj_Ker_weak_init},
\eqref{proj_Im_weak_init}. Indeed, the existence was proved by construction in the previous sections;
more precisely, function $m$ exists as the weak-$*$ limit of $(f^{\epsilon})_{\epsilon}$ and the existence
of functions $m_{1}^{\epsilon}$ and $n^{\epsilon}$ is guaranteed by Appendix A. Then, the linear transport
problem \eqref{eq_m} clearly has a unique solution. As for the linear
system \eqref{proj_Ker_weak}, \eqref{proj_Im_weak}, by classical results in linear partial
differential equations, we can prove that the solution of \eqref{proj_Ker_weak},
\eqref{proj_Im_weak}, \eqref{proj_Ker_weak_init}, \eqref{proj_Im_weak_init} is unique.

On the other hand, it can be easily shown that the classical Macro-Micro model \eqref{eq_m},
\eqref{proj_Ker_weak}, \eqref{proj_Im_weak}, \eqref{proj_Ker_weak_init}, \eqref{proj_Im_weak_init}
is equivalent to the original problem \eqref{Vlasov_rescaled_intro}.
\end{remark}

\begin{remark}[{\bf Conceiving Asymptotic-Preserving schemes}]\label{rem:apCMM}
In the light of Remark~\ref{classical-Mm}, the reformulated problem \eqref{eq_m},
\eqref{proj_Ker_weak}, \eqref{proj_Im_weak}, \eqref{proj_Ker_weak_init},
\eqref{proj_Im_weak_init} is Asymptotic-Preserving since formally, we obtain from 
\eqref{decompo_feps_weak} and \eqref{Vlasov_with_mm1n} that $m_{1}^{\epsilon} \to 0$ and
$n^{\epsilon} \to 0$ when $\epsilon \to 0$ and thus, the classical Macro-Micro model
reduces to \eqref{eq_m}. 

A time discretization of this reformulation can be derived following
\cite{Bennoune, Crouseilles, Lemou-Mieussens}. It is based on a
semi-implicit discretization in which the stiffest term in equation \eqref{proj_Im_weak}
is considered implicit to ensure stability as $\epsilon$ goes to $0$. The
so-obtained numerical scheme enjoys the Asymptotic-Preserving property since it is
consistent with \eqref{Vlasov_rescaled_intro} when $\epsilon > 0$ and consistent
with its weak-$*$ limit \eqref{eq_f_weak} when $\epsilon$ goes to $0$. This scheme
would play the role of problem $E_{\Delta z}^{\epsilon}\,u_{\Delta z}^{\epsilon} = 0$ in
the diagrams shown in \eqref{AP_diagram} and \eqref{diagram_TSAPS}.
%
\end{remark}

\section{Two-Scale Macro-Micro decomposition: preliminaries}
\label{SecTSMMDP} 
\setcounter{equation}{0}

The aim exposed in the first paragraph of Section \ref{Sec:Mmdeco} will now be changed
in order to obtain the Two-Scale limit model \eqref{eq_G} when $\epsilon\to0$. Eventually, 
we will see in Section \ref{SecTSMMvsMM}, that the Two-Scale Macro-Micro decomposition 
described in the following sections, will lead by integration in $\tau$ to the classical
Macro-Micro decomposition.

From now on, we assume that the input electric field and the input magnetic field do not
depend on $\epsilon$, \textit{i.e.} we take
\begin{equation}
\mathbf{E}^{\epsilon} = \mathbf{E} \, , \qquad \mathbf{B}^{\epsilon} = \mathbf{B} \, .
\end{equation}

\subsection{Complements on convergence results}

\indent Before entering in the core of the Two-Scale Macro-Micro decomposition, we integrate some knowledge which can be deduced from Fr\'enod, Raviart \& Sonnendr\"ucker \cite{Raviart} where the asymptotic expansion of $f^{\epsilon}$ is presented. Using hypotheses (\ref{hypotheses_f0}) and (\ref{hypotheses_EB}) and adding some regularity assumptions for $\mathbf{E}$ and $\mathbf{B}$, we can claim that
\begin{equation}\label{1stord_conv}
\cfrac{1}{\epsilon}\,\Big( f^{\epsilon}(t,\mathbf{x},\mathbf{v}) - F(t,\cfrac{t}{\epsilon},\mathbf{x},\mathbf{v})\Big) 
\textnormal{ ~~Two-Scale converges to~~ } \widetilde{F}(t,\tau,\mathbf{x},\mathbf{v}) \textnormal{ ~~as $\epsilon \to 0$,}
\end{equation}
and that $\widetilde{F}$ expresses in terms of other functions $\widetilde{G}$ and $l$ in the following way:
\begin{equation}
\widetilde{F}(t,\tau,\mathbf{x},\mathbf{v}) = \widetilde G\big(t,\mathbf{x},\ucar(\tau,\mathbf{v})\big) + l(t,\tau,\mathbf{x},\mathbf{v}) \, ,
\end{equation}
where
\begin{equation} \label{def_l}
\begin{split}
l(t,\tau,\mathbf{x},\mathbf{v}) &= \big( r(\tau+\cfrac{1}{4})-r(\cfrac{1}{4}) \big) \, \mathbf{v} \cdot \Nabla_{\mathbf{x}_{\perp}}G\big(t,\mathbf{x},\ucar(\tau,\mathbf{v})\big) \\
&\qquad + \Big[ \big( r(\tau+\cfrac{1}{4})-r(\cfrac{1}{4}) \big) \, \mathbf{E}(t,\mathbf{x}) + \ucar(\tau,\mathbf{v}) \times \Big(\big( r(\tau+\cfrac{1}{4})-r(\cfrac{1}{4}) \big) \, \mathbf{B}(t,\mathbf{x})\Big) \Big] \\
&\qquad \qquad \qquad \qquad \qquad \qquad \qquad \qquad \qquad \qquad \qquad \qquad \qquad \cdot \Nabla_{\mathbf{u}_{\perp}} G\big(t,\mathbf{x},\ucar(\tau,\mathbf{v})\big) \,,
\end{split}
\end{equation}
and where $\widetilde{G}$ satisfies an equation of the form
\begin{equation}\label{G_tilde}
\cfrac{\D \widetilde{G}}{\D t} + u_{||} \, \cfrac{\D \widetilde{G}}{\D x_{||}} + \big( \mathbf{E}_{||} + \mathbf{u} \cdot \mathbf{B}_{||} \big) \cdot \Nabla_{\mathbf{v}} \widetilde{G} = RHS(t,\mathbf{x},\mathbf{u},\mathbf{E},\mathbf{B},G) \, ,
\end{equation}
where the right hand side can be explicitly computed as in Theorem~4.2 of \cite{Raviart}.
In these equations, $G$ is the solution of (\ref{eq_G}), $r$ is defined in (\ref{mat_rotation}), and the gradients $\Nabla_{\mathbf{x}_{\perp}}$ and $\Nabla_{\mathbf{u}_{\perp}}$ are defined by
\begin{equation}
\Nabla_{\mathbf{x}_{\perp}} = \left(
\begin{array}{c}
0 \\ \cfrac{\D}{\D x_{2}} \\ \cfrac{\D}{\D x_{3}}
\end{array}
\right) \, , \qquad \Nabla_{\mathbf{u}_{\perp}} = \left(
\begin{array}{c}
0 \\ \cfrac{\D}{\D u_{2}} \\ \cfrac{\D}{\D u_{3}}
\end{array}
\right) \, .
\end{equation}

\subsection{Sought shape of $f^{\epsilon}$}

\indent As announced in the introduction, we want to use a shape given by (\ref{feps_decompo_intro})-(\ref{decompo_intro_2}) for $f^{\epsilon}$. \\
\indent We clearly have that operator
\begin{equation} \label{def_op}
\cfrac{\D}{\D\tau} + (\mathbf{v} \times \mathcal{M}) \cdot \Nabla_{\mathbf{v}} \, : \, L^{\infty}\big(0,T; L_{\#_{1}}^{\infty}(\R^+; L^{2}(\R^{6}))\big) \, \to \, L^{\infty}\big(0,T; L_{\#_{1}}^{\infty}(\R^+; L^{2}(\R^{6}))\big) \, ,
\end{equation}
is anti-symmetric, non bounded and with closed range (see Appendix B) and satisfies
\begin{equation} \label{decompo_orthogonal}
\textnormal{Ker}\big(\cfrac{\D}{\D\tau} + (\mathbf{v} \times \mathcal{M}) \cdot \Nabla_{\mathbf{v}} \big) \oplus \textnormal{Im}\big(\cfrac{\D}{\D\tau} + (\mathbf{v} \times \mathcal{M}) \cdot \Nabla_{\mathbf{v}} \big) = L^{\infty}\big(0,T; L_{\#_{1}}^{\infty}(\R^+; L^{2}(\R^{6}))\big) \, .
\end{equation}
On another hand, we have 
\begin{gather}
F \in \textnormal{Ker}\big(\cfrac{\D}{\D\tau} + (\mathbf{v} \times \mathcal{M}) \cdot \Nabla_{\mathbf{v}} \big).
\end{gather}
 Furthermore, we can easily prove that the function $l$ defined in \eqref{def_l} is such that
 \begin{gather}
 l\in \textnormal{Im}\big(\cfrac{\D}{\D\tau} + (\mathbf{v} \times \mathcal{M}) \cdot \Nabla_{\mathbf{v}} \big).
 \end{gather}
 Indeed, any regular function in $\textnormal{Ker}\big(\cfrac{\D}{\D\tau} + (\mathbf{v} \times \mathcal{M}) \cdot \Nabla_{\mathbf{v}} \big)$ 
 reads $\gamma(t,\mathbf{x},\ucar(\tau,\mathbf{v}))$ for some $\gamma \in L^{\infty}\big(0,T;L^{2}(\R^{6})\big)$. Multiplying $l(t,\tau,\mathbf{x},\mathbf{v})$ by $\gamma(t,\mathbf{x},\ucar(\tau,\mathbf{v}))$, integrating over $\R^{6} \times [0,1]$ and performing the change of variables
\begin{equation}
\mathbf{v} \, \mapsto \, \mathbf{u} = \ucar(\tau,\mathbf{v}) \, ,
\end{equation}
we obtain
\begin{equation}
\begin{split}
\int_{\R^{6}} \gamma(t,\mathbf{x},\mathbf{u})\, \Bigg\{ &\Big[ \Bigg( \int_{0}^{1} r(\tau-\cfrac{1}{4})-r(\cfrac{1}{4}) \, d\tau\Bigg) \, \mathbf{u} \Big] \cdot \Nabla_{\mathbf{x}_{\perp}} G(t,\mathbf{x},\mathbf{u}) \\
& - \Bigg[ \Bigg( \int_{0}^{1} r(\tau-\cfrac{1}{4})-r(\cfrac{1}{4})\, d\tau\Bigg)\, \mathbf{E}(t,\mathbf{x}) \\
&\qquad + \mathbf{u} \times \Big[ \Bigg( \int_{0}^{1} r(\tau-\cfrac{1}{4})-r(\cfrac{1}{4}) \, d\tau\Bigg)\,\mathbf{B}(t,\mathbf{x}) \Big] \Bigg] \cdot \Nabla_{\mathbf{u}_{\perp}}G(t,\mathbf{x},\mathbf{u}) \Bigg\} \, d\mathbf{x}\,d\mathbf{u} = 0 \, ,
\end{split}
\end{equation}
because the integrals over $[0,1]$ all vanish. \\

Now, we integrate every remarks we just did. We also notice that a decomposition of the type (\ref{feps_decompo_intro}) cannot be unique because of choices concerning the place where oscillations are put: in the superscript $\epsilon$ or in variable $\cfrac{t}{\epsilon}$. This will also prescribe what we set now: we look for 
\begin{equation}
\begin{array}{rclcl}
G_{1}^{\epsilon} &=& G_{1}^{\epsilon}(t,\mathbf{x},\mathbf{u})
&\in & L^{\infty}\big(0,T;L^{2}(\R^{6})\big), \\
k^{\epsilon} &=& k^{\epsilon}(t,\tau,\mathbf{x},\mathbf{v})
&\in & L^{\infty}\big(0,T;L_{\#_{1}}^{\infty}(\R^+; L^{2}(\R^{6}))\big),
\end{array}
\end{equation}
such that, for any $\tau$,
\begin{equation} \label{decompo_feps_TS}
\begin{split}
f^{\epsilon}(t,\mathbf{x},\mathbf{v}) &= G\big(t,\mathbf{x},\ucar(\tau,\mathbf{v})\big) + \epsilon \, G_{1}^{\epsilon}\big(t,\mathbf{x},\ucar(\tau,\mathbf{v})\big) \\
&\qquad + \epsilon\,l(t,\tau,\mathbf{x},\mathbf{v}) + \epsilon \left( \cfrac{\D k^{\epsilon}}{\D\tau}(t,\tau,\mathbf{x},\mathbf{v}) + (\mathbf{v}\times\mathcal{M}) \cdot \Nabla_{\mathbf{v}}k^{\epsilon}(t,\tau,\mathbf{x},\mathbf{v}) \right).
\end{split}
\end{equation}
By property \eqref{decompo_orthogonal} we know that $G_{1}^{\epsilon}$ and $k^{\epsilon}$ exist and are unique.
From now on, we shall denote
\begin{equation}
 (G \circ {\ucar})(t,\tau,\mathbf{x},\mathbf{v}) = G\big(t,\mathbf{x},\ucar(\tau,\mathbf{v})\big),
\end{equation}
leading, for instance, to the writing of $F = G \circ {\ucar}$. With this notation, we can claim that
\begin{equation}
\begin{array}{rcl}
(G \circ {\ucar}) + \epsilon \, (G_{1}^{\epsilon} \circ {\ucar}) &\in& \textnormal{Ker}\big(\cfrac{\D}{\D\tau} + (\mathbf{v} \times \mathcal{M}) \cdot \Nabla_{\mathbf{v}} \big), \\
\epsilon\,l + \epsilon \Big(\cfrac{\D k^{\epsilon}}{\D\tau} + (\mathbf{v}\times\mathcal{M}) \cdot \Nabla_{\mathbf{v}}k^{\epsilon} \Big)&\in& \textnormal{Im}\big(\cfrac{\D}{\D\tau} + (\mathbf{v} \times \mathcal{M}) \cdot \Nabla_{\mathbf{v}} \big),
\end{array}
\end{equation}
and, as a consequence of (\ref{decompo_orthogonal}), decomposition (\ref{decompo_feps_TS}) exists and is unique since $G$ and $l$ have already been identified as the unique couple satisfying (\ref{eq_G}) and (\ref{def_l}). \\

\indent In the previous section, when building the usual Macro-Micro decomposition, we projected the equation satisfied by $f^{\epsilon}$ onto the kernel and the range of the involved operator. Here, this cannot be done. Indeed, here, the involved operator is $\frac{\D}{\D\tau} + (\mathbf{v} \times \mathcal{M}) \cdot \Nabla_{\mathbf{v}}$ and the function which has to be used in equation (\ref{Vlasov_rescaled_intro}) in order to extract some informations on $G_{1}^{\epsilon}$ and $k^{\epsilon}$ is
\begin{equation}\label{moregeneral_decompo_feps_TS}
\begin{split}
f^{\epsilon}(t,\mathbf{x},\mathbf{v}) &= (G\circ{\ucar})(t,\cfrac{t}{\epsilon},\mathbf{x},\mathbf{v})
+ \epsilon \, (G_{1}^{\epsilon}\circ{\ucar})(t,\cfrac{t}{\epsilon},\mathbf{x},\mathbf{v}) \\
&\qquad + \epsilon\,l(t,\cfrac{t}{\epsilon},\mathbf{x},\mathbf{v}) + \epsilon\left(
\cfrac{\D k^{\epsilon}}{\D\tau}(t,\cfrac{t}{\epsilon},\mathbf{x},\mathbf{v}) + (\mathbf{v}\times
\mathcal{M}) \cdot \Nabla_{\mathbf{v}}k^{\epsilon}(t,\cfrac{t}{\epsilon},\mathbf{x},\mathbf{v})\right),
\end{split}
\end{equation}
which corresponds to \eqref{decompo_feps_TS} with $\tau = \frac{t}{\epsilon}$.

\subsection{The pathway towards Two-Scale Macro and Micro equations}

In this section we summarize the way we follow to obtain the equations for $G_{1}^{\epsilon}$
and $k^{\epsilon}$. They make up the Two-Scale Macro and Micro equations.
Once $\tau$ was replaced by $\cfrac{t}{\epsilon}$ in \eqref{decompo_feps_TS}, we face
with functional spaces where
operator (\ref{def_op}) makes no sense. In order to bypass this difficulty, we use a
weak formulation of (\ref{Vlasov_rescaled_intro}) with oscillating test functions
\begin{equation}
(\psi)^{\epsilon} = (\psi)^{\epsilon}(t,\mathbf{x},\mathbf{v}) = \psi(t,\frac{t}{\epsilon},
\mathbf{x},\mathbf{v}),
\end{equation}
where $\psi = \psi(t,\tau,\mathbf{x},\mathbf{v})$ is regular, $1$-periodic in $\tau$ and with compact support on $[0,T) \times \R^{3} \times \R^{3}$ for any fixed $\tau\in [0,1]$. Writing
\begin{equation} \label{link_F1G1}
(F_{1}^{\epsilon})^{\epsilon} = (F_{1}^{\epsilon})^{\epsilon}(t,\mathbf{x},\mathbf{v}) = F_{1}^{\epsilon}(t,\frac{t}{\epsilon},\mathbf{x},\mathbf{v}) = (G_{1}^{\epsilon} \circ {\ucar})(t,\frac{t}{\epsilon},\mathbf{x},\mathbf{v}) \, ,
\end{equation}
and using the same convention for other functions, the weak formulation with oscillating test functions reads
\begin{equation} \label{Vlasov_decomposed_weak_TS}
\begin{split}
\int_{0}^{T} \int_{\R^{6}} \Big[ (F)^{\epsilon} &+ \epsilon\,(F_{1}^{\epsilon})^{\epsilon} + \epsilon\,(l)^{\epsilon} + \epsilon (h^{\epsilon})^{\epsilon}\Big] \, \Big[ \big( \cfrac{\D \psi}{\D t}\big)^{\epsilon} + \cfrac{1}{\epsilon}\, \big( \cfrac{\D \psi}{\D \tau}\big)^{\epsilon} + \mathbf{v} \cdot (\Nabla_{\mathbf{x}}\psi)^{\epsilon} \\
&+ \big( \mathbf{E} + \mathbf{v} \times (\mathbf{B} + \cfrac{\mathcal{M}}{\epsilon})\big) \cdot (\Nabla_{\mathbf{v}}\psi)^{\epsilon} \Big] \, d\mathbf{x}\,d\mathbf{v} \, dt = -\int_{\R^{6}} f_{0}(\mathbf{x},\mathbf{v}) \, \psi(0,0,\mathbf{x},\mathbf{v}) \, d\mathbf{x} \, d\mathbf{v} \, ,
\end{split}
\end{equation}
where $F$ is linked with $G$ by (\ref{link_FG}), $F_{1}^{\epsilon}$ with $G_{1}^{\epsilon}$ by (\ref{link_F1G1}), $l$ is defined in (\ref{def_l}) and $h^{\epsilon}$ is linked with $k^{\epsilon}$ by
\begin{equation} \label{link_kepsheps}
h^{\epsilon} = \cfrac{\D k^{\epsilon}}{\D\tau} + (\mathbf{v} \times \mathcal{M})\cdot\Nabla_{\mathbf{v}}k^{\epsilon} \, .
\end{equation}
The way to deduce the equation for $G_{1}^{\epsilon}$ and for $k^{\epsilon}$ consists in using, in a first computation, test functions $\psi \in \textnormal{Ker}\big(\cfrac{\D}{\D\tau} + (\mathbf{v} \times \mathcal{M}) \cdot \Nabla_{\mathbf{v}} \big)$, \textit{i.e.} of the form
\begin{equation} \label{def_psi_Ker}
\psi = (\gamma \circ {\ucar}) \, ,
\end{equation}
for regular functions $\gamma = \gamma(t,\mathbf{x},\mathbf{u})$ compactly supported in $[0,T) \times \R^{3} \times \R^{3}$. Then, in a second computation, test functions belonging to $\textnormal{Im}\big(\cfrac{\D}{\D\tau} + (\mathbf{v} \times \mathcal{M}) \cdot \Nabla_{\mathbf{v}} \big)$, or writing
\begin{equation} \label{def_psi_Im}
\psi = \cfrac{\D \kappa}{\D \tau} + (\mathbf{v} \times \mathcal{M}) \cdot \Nabla_{\mathbf{v}}\kappa \, ,
\end{equation}
for some regular functions $\kappa = \kappa(t,\tau,\mathbf{x},\mathbf{v})$ with compact support in $[0,T) \times \R^{3} \times \R^{3}$ for every $\tau \in [0,1]$, will be chosen.

\section{The Two-Scale Macro equation}
\label{SecTMaE} 
\setcounter{equation}{0}

We begin by going further in the writing of weak formulation (\ref{Vlasov_decomposed_weak_TS}) by using the link between $G$ and $l$. We recall this weak formula:
\begin{equation} \label{Proj_Ker_TS_1}
\begin{split}
\int_{0}^{T} &\int_{\R^{6}} (F)^{\epsilon}\, \Big[ \cfrac{\D (\psi)^{\epsilon}}{\D t} + \mathbf{v} \cdot \Nabla_{\mathbf{x}}(\psi)^{\epsilon} + \big( \mathbf{E} + \mathbf{v} \times (\mathbf{B} + \cfrac{\mathcal{M}}{\epsilon})\big) \cdot \Nabla_{\mathbf{v}}(\psi)^{\epsilon} \Big] \, d\mathbf{x}\,d\mathbf{v} \, dt \\
&+ \epsilon\, \int_{0}^{T} \int_{\R^{6}} (F_{1}^{\epsilon})^{\epsilon}\, \Big[ (\cfrac{\D \psi}{\D t})^{\epsilon} + \cfrac{1}{\epsilon}\,(\cfrac{\D\psi}{\D\tau})^{\epsilon} + \mathbf{v} \cdot (\Nabla_{\mathbf{x}}\psi)^{\epsilon} + \big( \mathbf{E} + \mathbf{v} \times (\mathbf{B} + \cfrac{\mathcal{M}}{\epsilon})\big) \cdot (\Nabla_{\mathbf{v}}\psi)^{\epsilon} \Big] \, d\mathbf{x}\,d\mathbf{v} \, dt \\
&+ \epsilon\, \int_{0}^{T} \int_{\R^{6}} (l)^{\epsilon}\, \Big[ \cfrac{\D (\psi)^{\epsilon}}{\D t} + (\mathbf{v} \times \cfrac{\mathcal{M}}{\epsilon}) \cdot \Nabla_{\mathbf{v}}(\psi)^{\epsilon} \Big] \, d\mathbf{x}\,d\mathbf{v} \, dt \\
&+\epsilon\, \int_{0}^{T} \int_{\R^{6}} (l)^{\epsilon} \, \Big[ \mathbf{v} \cdot (\Nabla_{\mathbf{x}}\psi)^{\epsilon} + (\mathbf{E}+\mathbf{v} \times \mathbf{B}) \cdot (\Nabla_{\mathbf{v}}\psi)^{\epsilon} \Big] \, d\mathbf{x} \, d\mathbf{v} \, dt \\
&+\epsilon\, \int_{0}^{T} \int_{\R^{6}} (h^{\epsilon})^{\epsilon}\, \Big[ (\cfrac{\D \psi}{\D t})^{\epsilon} + \cfrac{1}{\epsilon}\,(\cfrac{\D\psi}{\D\tau})^{\epsilon} + \mathbf{v} \cdot (\Nabla_{\mathbf{x}}\psi)^{\epsilon} + \big( \mathbf{E} + \mathbf{v} \times (\mathbf{B} + \cfrac{\mathcal{M}}{\epsilon})\big) \cdot (\Nabla_{\mathbf{v}}\psi)^{\epsilon} \Big] \, d\mathbf{x}\,d\mathbf{v} \, dt \\
&\qquad \qquad = -\int_{\R^{6}} f_{0}(\mathbf{x},\mathbf{v}) \, \psi(0,0,\mathbf{x},\mathbf{v}) \, d\mathbf{x} \, d\mathbf{v} \, .
\end{split}
\end{equation}
Performing an integration by parts in the first term of \eqref{Proj_Ker_TS_1} gives
\begin{equation} \label{term_1_Ker}
\begin{split}
\int_{0}^{T} &\int_{\R^{6}} (F)^{\epsilon}\, \Big[ \cfrac{\D (\psi)^{\epsilon}}{\D t} + \mathbf{v} \cdot \Nabla_{\mathbf{x}}(\psi)^{\epsilon} + \big( \mathbf{E} + \mathbf{v} \times (\mathbf{B} + \cfrac{\mathcal{M}}{\epsilon})\big) \cdot \Nabla_{\mathbf{v}}(\psi)^{\epsilon} \Big] \, d\mathbf{x}\,d\mathbf{v} \, dt \\
&= -\int_{0}^{T} \int_{\R^{6}} \Big[ (\cfrac{\D F}{\D t})^{\epsilon} + \cfrac{1}{\epsilon}\,(\cfrac{\D F}{\D\tau})^{\epsilon} + \mathbf{v} \cdot (\Nabla_{\mathbf{x}}F)^{\epsilon} + \big( \mathbf{E} + \mathbf{v} \times (\mathbf{B} + \cfrac{\mathcal{M}}{\epsilon})\big) \cdot (\Nabla_{\mathbf{v}}F)^{\epsilon} \Big] \, (\psi)^{\epsilon}\, d\mathbf{x} \, d\mathbf{v} \, dt \\
&\qquad \qquad - \int_{\R^{6}} F(0,0,\mathbf{x},\mathbf{v}) \, \psi(0,0,\mathbf{x},\mathbf{v}) \, d\mathbf{x} \, d\mathbf{v} \\
&= -\int_{0}^{T} \int_{\R^{6}} \Big[ (\cfrac{\D F}{\D t})^{\epsilon} + \mathbf{v} \cdot (\Nabla_{\mathbf{x}}F)^{\epsilon} + (\mathbf{E} + \mathbf{v} \times \mathbf{B}) \cdot (\Nabla_{\mathbf{v}}F)^{\epsilon} \Big] \, (\psi)^{\epsilon}\, d\mathbf{x} \, d\mathbf{v} \, dt \\
&\qquad \qquad - \int_{\R^{6}} f_{0}(\mathbf{x},\mathbf{v}) \, \psi(0,0,\mathbf{x},\mathbf{v}) \, d\mathbf{x} \, d\mathbf{v} \, ,
\end{split}
\end{equation}
since $F \in \textnormal{Ker}(\cfrac{\D}{\D\tau}+(\mathbf{v} \times \mathcal{M})\cdot\Nabla_{\mathbf{v}})$. We invoke the link between $F$ and $G$, and we consider the change of variables 
\begin{gather}
\label{DefFctTest222}
\mathbf{v}\mapsto\mathbf{u} = \ucar(\frac{t}{\epsilon},\mathbf{v}) 
\text{ which is equivalent to } \mathbf{v} = \ucar(-\frac{t}{\epsilon},\mathbf{u})={\ucar}^{-1}(\frac{t}{\epsilon},\mathbf{u}). 
\end{gather}
We also use the following notation
\begin{equation}
(\psi \circ{\ucar}^{-1})(t,\tau,\mathbf{x},\mathbf{u})= \psi\big(t,\tau,\mathbf{x},{\ucar}^{-1}
(\tau,\mathbf{u})\big) = \psi\big(t,\tau,\mathbf{x},\ucar(-\tau,\mathbf{u})\big) \, .
\end{equation}
Since a direct computation yields
\begin{equation}
(\Nabla_{\mathbf{v}}F)^{\epsilon} = \Nabla_{\mathbf{v}}(F)^{\epsilon} = \Nabla_{\mathbf{v}}(G \circ {\ucar})^{\epsilon} = r(-\frac{t}{\epsilon}) \big((\Nabla_{\mathbf{u}} G) \circ {\ucar}\big)^{\epsilon} \, ,
\end{equation}
equation (\ref{term_1_Ker}) reads
\begin{equation}
\begin{split}
\int_{0}^{T} &\int_{\R^{6}} (F)^{\epsilon}\, \Big[ \cfrac{\D (\psi)^{\epsilon}}{\D t} + \mathbf{v} \cdot \Nabla_{\mathbf{x}}(\psi)^{\epsilon} + \big( \mathbf{E} + \mathbf{v} \times (\mathbf{B} + \cfrac{\mathcal{M}}{\epsilon})\big) \cdot \Nabla_{\mathbf{v}}(\psi)^{\epsilon} \Big] \, d\mathbf{x}\,d\mathbf{v} \, dt \\
&= - \int_{0}^{T} \int_{\R^{6}} \Big[ \cfrac{\D G}{\D t} + \big( r(-\frac{t}{\epsilon})\,\mathbf{u}\big) \cdot \Nabla_{\mathbf{x}} G + \big( r(\frac{t}{\epsilon})\, \big[ \mathbf{E} + \big(r(-\frac{t}{\epsilon})\,\mathbf{u}\big) \times \mathbf{B} \big] \cdot \Nabla_{\mathbf{u}}G \Big] \\
&\qquad \qquad \qquad \qquad \qquad \qquad \qquad \qquad \qquad \qquad \qquad \qquad \qquad \qquad \qquad (\psi \circ{\ucar}^{-1})^{\epsilon} \, d\mathbf{x} \, d\mathbf{u} \, dt \\
&\qquad - \int_{\R^{6}} f_{0}(\mathbf{x},\mathbf{v})\,\psi(0,0,\mathbf{x},\mathbf{v}) \, d\mathbf{x} \, d\mathbf{v} \\
&= - \int_{0}^{T} \int_{\R^{6}} \Big[ \cfrac{\D G}{\D t} + \big( r(-\frac{t}{\epsilon})\,\mathbf{u}\big) \cdot \Nabla_{\mathbf{x}} G + \big[r(\frac{t}{\epsilon})\,\mathbf{E} + \mathbf{u} \times \big(r(\frac{t}{\epsilon})\,\mathbf{B}\big) \big] \cdot \Nabla_{\mathbf{u}}G \Big] \, (\psi \circ{\ucar}^{-1})^{\epsilon} \, d\mathbf{x} \, d\mathbf{u} \, dt \\
&\qquad - \int_{\R^{6}} f_{0}(\mathbf{x},\mathbf{v})\,\psi(0,0,\mathbf{x},\mathbf{v}) \, d\mathbf{x} \, d\mathbf{v} \, .
\end{split}
\end{equation}
Using now equation (\ref{eq_G}) satisfied by $G$, we finally obtain
\begin{equation} \label{term_1_Ker_final}
\begin{split}
\int_{0}^{T} &\int_{\R^{6}} (F)^{\epsilon}\, \Big[ \cfrac{\D (\psi)^{\epsilon}}{\D t} + \mathbf{v} \cdot \Nabla_{\mathbf{x}}(\psi)^{\epsilon} + \big( \mathbf{E} + \mathbf{v} \times (\mathbf{B} + \cfrac{\mathcal{M}}{\epsilon})\big) \cdot \Nabla_{\mathbf{v}}(\psi)^{\epsilon} \Big] \, d\mathbf{x}\,d\mathbf{v} \, dt \\
&= - \int_{0}^{T} \int_{\R^{6}} \Bigg(\big(r(-\frac{t}{\epsilon})\,\mathbf{u}\big) \cdot \Nabla_{\mathbf{x}_{\perp}}G + \Big[ \big(r(\frac{t}{\epsilon})\,\mathbf{E}\big)_{\perp} + \big[\mathbf{u} \times \big(r(\frac{t}{\epsilon})\,\mathbf{B}\big)\big]_{\perp}\Big] \cdot \Nabla_{\mathbf{u}_{\perp}}G \Bigg) \\
&\qquad \qquad \qquad \qquad \qquad \qquad \qquad \qquad \qquad \qquad \qquad \qquad \qquad \qquad \qquad (\psi \circ{\ucar}^{-1})^{\epsilon} \, d\mathbf{x} \, d\mathbf{u} \, dt \\
&\qquad - \int_{\R^{6}} f_{0}(\mathbf{x},\mathbf{v})\,\psi(0,0,\mathbf{x},\mathbf{v}) \, d\mathbf{x} \, d\mathbf{v} \, .
\end{split}
\end{equation}
Notice that the last term containing $f_{0}$ will vanish with the right hand side member of (\ref{Proj_Ker_TS_1}). \\

\indent Now, we study the third term of (\ref{Proj_Ker_TS_1}). For this purpose, we first notice that
\begin{equation}
\begin{split}
\cfrac{\D (l \circ{\ucar}^{-1})}{\D\tau} &= \big((\cfrac{\D l}{\D\tau}) \, \circ{\ucar}^{-1}\big) + \big( (r(-\tau)\,\mathbf{u}) \times \mathcal{M}\big) \cdot \big((\Nabla_{\mathbf{v}}l)\circ{\ucar}^{-1}\big) \\
&= - \big(r(-\tau)\,\mathbf{u}\big)\cdot \Nabla_{\mathbf{x}_{\perp}}G - \big[ r(\tau)\,\mathbf{E} + \mathbf{u} \times \big(r(\tau)\,\mathbf{B}\big)\big] \cdot \Nabla_{\mathbf{u}_{\perp}}G \, ,
\end{split}
\end{equation}
hence, integrating by parts the third term of (\ref{Proj_Ker_TS_1}) and making the change of variables defined by (\ref{DefFctTest222}), we get
\begin{equation} \label{term_3_Ker_final}
\begin{split}
\epsilon\, \int_{0}^{T} &\int_{\R^{6}} (l)^{\epsilon}\, \Big[ \cfrac{\D (\psi)^{\epsilon}}{\D t} + (\mathbf{v} \times \cfrac{\mathcal{M}}{\epsilon}) \cdot \Nabla_{\mathbf{v}}(\psi)^{\epsilon} \Big] \, d\mathbf{x}\,d\mathbf{v} \, dt \\
&= - \int_{0}^{T} \int_{\R^{6}} \Big[ \epsilon\,(\cfrac{\D l}{\D t})^{\epsilon} + (\cfrac{\D l}{\D\tau})^{\epsilon} + (\mathbf{v} \times \mathcal{M}) \cdot (\Nabla_{\mathbf{v}}l)^{\epsilon} \Big] \, (\psi)^{\epsilon} \, d\mathbf{x}\,d\mathbf{v} \, dt - \int_{\R^{6}} (l)^{\epsilon}(0,\mathbf{x},\mathbf{v})\,(\psi)^{\epsilon}(0,\mathbf{x},\mathbf{v}) \, d\mathbf{x}\,d\mathbf{v} \, , \\
&= - \int_{0}^{T} \int_{\R^{6}} \Big[ \epsilon\, \big(\cfrac{\D (l\circ{\ucar}^{-1})}{\D t}\big)^{\epsilon} + \big(\cfrac{\D (l\circ{\ucar}^{-1})}{\D \tau}\big)^{\epsilon} \Big] \, (\psi \circ{\ucar}^{-1})^{\epsilon} \, d\mathbf{x}\,d\mathbf{u} \, dt \\
&= - \int_{0}^{T} \int_{\R^{6}} \epsilon\, \big(\cfrac{\D (l\circ{\ucar}^{-1})}{\D t}\big)^{\epsilon} \, (\psi \circ{\ucar}^{-1})^{\epsilon} \, d\mathbf{x}\,d\mathbf{u} \, dt \\
&\qquad + \int_{0}^{T} \int_{\R^{6}} \Big[ \big(r(-\frac{t}{\epsilon})\,\mathbf{u}\big)\cdot \Nabla_{\mathbf{x}_{\perp}}G + \big[ r(\frac{t}{\epsilon})\,\mathbf{E} + \mathbf{u} \times \big(r(\cfrac{t}{\epsilon})\,\mathbf{B}\big)\big] \cdot \Nabla_{\mathbf{u}_{\perp}}G \Big]\, (\psi \circ{\ucar}^{-1})^{\epsilon} \, d\mathbf{x}\,d\mathbf{u} \, dt \, ,
\end{split}
\end{equation}
since $l(0,0,\mathbf{x},\mathbf{v}) = 0$ (see (\ref{def_l})). The last term of (\ref{term_3_Ker_final})
and the first term in the right hand side of (\ref{term_1_Ker_final}) will cancel each other. \\

\indent Hence (\ref{Proj_Ker_TS_1}) has to be replaced by
\begin{equation} \label{Proj_Ker_TS_2}
\begin{split}
- &\int_{0}^{T} \int_{\R^{6}} \epsilon\, \big(\cfrac{\D (l\circ{\ucar}^{-1})}{\D t}\big)^{\epsilon} \, (\psi \circ{\ucar}^{-1})^{\epsilon} \, d\mathbf{x}\,d\mathbf{u} \, dt \\
&+ \epsilon\, \int_{0}^{T} \int_{\R^{6}} (F_{1}^{\epsilon})^{\epsilon}\, \Big[ (\cfrac{\D \psi}{\D t})^{\epsilon} + \cfrac{1}{\epsilon}\,(\cfrac{\D\psi}{\D\tau})^{\epsilon} + \mathbf{v} \cdot (\Nabla_{\mathbf{x}}\psi)^{\epsilon} + \big( \mathbf{E} + \mathbf{v} \times (\mathbf{B} + \cfrac{\mathcal{M}}{\epsilon})\big) \cdot (\Nabla_{\mathbf{v}}\psi)^{\epsilon} \Big] \, d\mathbf{x}\,d\mathbf{v} \, dt \\
&+\epsilon\, \int_{0}^{T} \int_{\R^{6}} (l)^{\epsilon} \, \Big[ \mathbf{v} \cdot (\Nabla_{\mathbf{x}}\psi)^{\epsilon} + (\mathbf{E}+\mathbf{v} \times \mathbf{B}) \cdot (\Nabla_{\mathbf{v}}\psi)^{\epsilon} \Big] \, d\mathbf{x} \, d\mathbf{v}\, dt \\
&+\epsilon\, \int_{0}^{T} \int_{\R^{6}} (h^{\epsilon})^{\epsilon}\, \Big[ (\cfrac{\D \psi}{\D t})^{\epsilon} + \cfrac{1}{\epsilon}\,(\cfrac{\D\psi}{\D\tau})^{\epsilon} + \mathbf{v} \cdot (\Nabla_{\mathbf{x}}\psi)^{\epsilon} + \big( \mathbf{E} + \mathbf{v} \times (\mathbf{B} + \cfrac{\mathcal{M}}{\epsilon})\big) \cdot (\Nabla_{\mathbf{v}}\psi)^{\epsilon} \Big] \, d\mathbf{x}\,d\mathbf{v} \, dt = 0 \, .
\end{split}
\end{equation}
Now, choosing $\psi$ as in (\ref{def_psi_Ker}) yields
\begin{equation} \label{Proj_Ker_TS_3}
\begin{split}
\epsilon\,\int_{0}^{T} &\int_{\R^{6}} (G_{1}^{\epsilon} \circ{\ucar})^{\epsilon} \, \Big[ \big((\cfrac{\D \gamma}{\D t})\circ {\ucar}\big)^{\epsilon} + \mathbf{v} \cdot \big( (\Nabla_{\mathbf{x}}\gamma) \circ {\ucar}\big)^{\epsilon} + (\mathbf{E} +\mathbf{v} \times \mathbf{B}) \cdot \Nabla_{\mathbf{v}}(\gamma \circ {\ucar})^{\epsilon} \Big] \, d\mathbf{x} \, d\mathbf{v} \, dt \\
&- \int_{0}^{T} \int_{\R^{6}} \epsilon\, \big(\cfrac{\D (l\circ{\ucar}^{-1})}{\D t}\big)^{\epsilon} \, \gamma \, d\mathbf{x}\,d\mathbf{u} \, dt + \int_{0}^{T} \int_{\R^{6}} \epsilon\,(h^{\epsilon})^{\epsilon} \, \big( \cfrac{\D (\gamma\circ{\ucar})}{\D t}\big)^{\epsilon} \, d\mathbf{x}\,d\mathbf{v} \, dt \\
&+ \int_{0}^{T} \int_{\R^{6}} \big( \epsilon\, (l)^{\epsilon} + \epsilon\,(h^{\epsilon})^{\epsilon}\big)\, \Big[ \mathbf{v} \cdot \big( \Nabla_{\mathbf{x}}(\gamma \circ {\ucar})\big)^{\epsilon} + (\mathbf{E} + \mathbf{v} \times \mathbf{B}) \cdot \big( \Nabla_{\mathbf{v}}(\gamma \circ {\ucar})\big)^{\epsilon} \Big] \, d\mathbf{x}\,d\mathbf{v} \, dt = 0 \, .
\end{split}
\end{equation}
Making in the first, third and fourth terms the change of variables $\mathbf{v} \mapsto \ucar(\cfrac{t}{\epsilon},\mathbf{v})$, replacing $h^{\epsilon}$ by its expression in terms of $k^{\epsilon}$ and dividing (\ref{Proj_Ker_TS_3}) by $\epsilon$ finally gives
\begin{equation} \label{Proj_Ker_TS_final}
\begin{split}
\int_{0}^{T} &\int_{\R^{6}} G_{1}^{\epsilon} \, \Big[ \cfrac{\D\gamma}{\D t} + \big(r(-\frac{t}{\epsilon})\,\mathbf{u}\big) \cdot \Nabla_{\mathbf{x}} \gamma + \big[ r(\frac{t}{\epsilon})\,\mathbf{E} + \mathbf{u} \times \big(r(\frac{t}{\epsilon})\,\mathbf{B}\big) \big] \cdot \Nabla_{\mathbf{u}}\gamma \Big] \, d\mathbf{x}\,d\mathbf{u} \, dt \\
&- \int_{0}^{T} \int_{\R^{6}} \big(\cfrac{\D (l\circ{\ucar}^{-1})}{\D t}\big)^{\epsilon} \, \gamma \, d\mathbf{x}\,d\mathbf{u} \, dt \\
&+ \int_{0}^{T} \int_{\R^{6}} \Big[ \big( (\cfrac{\D k^{\epsilon}}{\D \tau}) \circ{\ucar}^{-1}\big)^{\epsilon} + \big(r(-\frac{t}{\epsilon})\,\mathcal{M}\big) \cdot \big( (\Nabla_{\mathbf{v}}k^{\epsilon}) \circ{\ucar}^{-1}\big)^{\epsilon} \Big] \, \cfrac{\D\gamma}{\D t} \, d\mathbf{x}\,d\mathbf{u} \, dt \\
&+ \int_{0}^{T} \int_{\R^{6}} \Big[ (l \circ{\ucar}^{-1})^{\epsilon} + \big( (\cfrac{\D k^{\epsilon}}{\D \tau}) \circ{\ucar}^{-1}\big)^{\epsilon} + \big(r(-\frac{t}{\epsilon})\,\mathcal{M}\big) \cdot \big( (\Nabla_{\mathbf{v}}k^{\epsilon}) \circ{\ucar}^{-1}\big)^{\epsilon} \Big] \\
&\qquad \qquad \qquad \times \Big[ \big( r(-\frac{t}{\epsilon})\,\mathbf{u}\big) \cdot \Nabla_{\mathbf{x}}\gamma + \big[ r(\frac{t}{\epsilon})\,\mathbf{E} + \mathbf{u} \times \big( r(\frac{t}{\epsilon})\,\mathbf{B}\big) \big] \cdot \Nabla_{\mathbf{u}}\gamma \Big] \, d\mathbf{x} \, d\mathbf{u} \, dt = 0 \, ,
\end{split}
\end{equation}
which is the Two-Scale Macro equation of the model.

\section{The Two-Scale Micro equation}
\label{SecTMiE} 
\setcounter{equation}{0}
\indent In this section, we use in the weak formulation (\ref{Vlasov_decomposed_weak_TS}) oscillating test functions defined by (\ref{def_psi_Im}). The computation leading to formula (\ref{Proj_Ker_TS_2}) is valid for any oscillating function $\psi$ so we use it as a starting point for finding the Micro equation. \\

\indent Recalling that $F_{1}^{\epsilon} \in \textnormal{Ker}(\cfrac{\D}{\D\tau}+(\mathbf{v} \times \mathcal{M})\cdot \Nabla_{\mathbf{v}})$, the second term in (\ref{Proj_Ker_TS_2}) yields
\begin{equation}
\begin{split}
\epsilon\, \int_{0}^{T} &\int_{\R^{6}} (F_{1}^{\epsilon})^{\epsilon}\, \Big[ (\cfrac{\D \psi}{\D t})^{\epsilon} + \cfrac{1}{\epsilon}\,(\cfrac{\D\psi}{\D\tau})^{\epsilon} + \mathbf{v} \cdot (\Nabla_{\mathbf{x}}\psi)^{\epsilon} + \big( \mathbf{E} + \mathbf{v} \times (\mathbf{B} + \cfrac{\mathcal{M}}{\epsilon})\big) \cdot (\Nabla_{\mathbf{v}}\psi)^{\epsilon} \Big] \, d\mathbf{x}\,d\mathbf{v} \, dt \\
&= \epsilon\, \int_{0}^{T} \int_{\R^{6}} (F_{1}^{\epsilon})^{\epsilon}\, \Big[ \cfrac{\D (\psi)^{\epsilon}}{\D t} + (\mathbf{v} \times \cfrac{\mathcal{M}}{\epsilon}) \cdot \Nabla_{\mathbf{v}}(\psi)^{\epsilon} \Big] \, d\mathbf{x}\,d\mathbf{v} \, dt \\
&\qquad + \epsilon\, \int_{0}^{T} \int_{\R^{6}} (F_{1}^{\epsilon})^{\epsilon}\, \Big[ \mathbf{v} \cdot (\Nabla_{\mathbf{x}}\psi)^{\epsilon} + (\mathbf{E} + \mathbf{v} \times \mathbf{B}) \cdot (\Nabla_{\mathbf{v}}\psi)^{\epsilon}\Big] \, d\mathbf{x}\,d\mathbf{v} \, dt \\
&= - \epsilon \, \int_{0}^{T} \int_{\R^{6}} (\cfrac{\D F_{1}^{\epsilon}}{\D t})^{\epsilon} \, (\psi)^{\epsilon} \, d\mathbf{x}\,d\mathbf{v} \, dt + \epsilon\, \int_{\R^{6}} F_{1}^{\epsilon}(0,0,\mathbf{x},\mathbf{v})\, \psi(0,0,\mathbf{x},\mathbf{v})\, d\mathbf{x}\,d\mathbf{v} \\
&\qquad + \epsilon\, \int_{0}^{T} \int_{\R^{6}} (F_{1}^{\epsilon})^{\epsilon}\, \Big[ \mathbf{v} \cdot (\Nabla_{\mathbf{x}}\psi)^{\epsilon} + (\mathbf{E} + \mathbf{v} \times \mathbf{B}) \cdot (\Nabla_{\mathbf{v}}\psi)^{\epsilon}\Big] \, d\mathbf{x}\,d\mathbf{v} \, dt \\
&= - \epsilon \, \int_{0}^{T} \int_{\R^{6}} (\cfrac{\D F_{1}^{\epsilon}}{\D t})^{\epsilon} \Big[ (\cfrac{\D\kappa}{\D\tau})^{\epsilon} + (\mathbf{v} \times \mathcal{M}) \cdot (\Nabla_{\mathbf{v}}\kappa)^{\epsilon} \Big] \, d\mathbf{x} \, d\mathbf{v} \, dt \\
&\qquad +\epsilon\, \int_{\R^{6}} F_{1}^{\epsilon}(0,0,\mathbf{x},\mathbf{v})\, \Big[ \cfrac{\D\kappa}{\D\tau}(0,0,\mathbf{x},\mathbf{v}) + (\mathbf{v} \times \mathcal{M}) \cdot \Nabla_{\mathbf{v}}\kappa(0,0,\mathbf{x},\mathbf{v})\Big] \, d\mathbf{x}\,d\mathbf{v} \\
&\qquad + \epsilon\, \int_{0}^{T} \int_{\R^{6}} (F_{1}^{\epsilon})^{\epsilon} \, \Bigg[ \mathbf{v} \cdot (\cfrac{\D \Nabla_{\mathbf{x}}\kappa}{\D \tau})^{\epsilon} + (\mathbf{v} \times \mathcal{M}) \cdot \big((\Nabla_{\mathbf{x}}\Nabla_{\mathbf{v}}\kappa)^{\epsilon}\, \mathbf{v}\big) \\
&\qquad \qquad \qquad \qquad \qquad + (\mathbf{E} + \mathbf{v} \times \mathbf{B}) \cdot (\cfrac{\D \Nabla_{\mathbf{v}}\kappa}{\D \tau})^{\epsilon} + (\mathbf{E} \times \mathcal{M} + \mathbf{v} \times \mathbf{B} \times \mathcal{M}) \cdot (\Nabla_{\mathbf{v}}\kappa)^{\epsilon} \\
&\qquad \qquad \qquad \qquad \qquad + (\mathbf{E}+\mathbf{v}\times\mathcal{M}) \cdot \big( (\Nabla_{\mathbf{v}}^{\;2}\kappa)^{\epsilon}(\mathbf{v}\times\mathcal{M})\big) \Bigg] \, d\mathbf{x}\,d\mathbf{v} \, dt \, .
\end{split}
\end{equation}
The first term of (\ref{Proj_Ker_TS_2}) reads
\begin{equation}
\begin{split}
- &\int_{0}^{T} \int_{\R^{6}} \epsilon\, \big(\cfrac{\D (l\circ{\ucar}^{-1})}{\D t}\big)^{\epsilon} \, (\psi \circ{\ucar}^{-1})^{\epsilon} \, d\mathbf{x}\,d\mathbf{v} \, dt \\
&= -\epsilon \, \int_{0}^{T} \int_{\R^{6}} \big(\cfrac{\D (l\circ{\ucar}^{-1})}{\D t}\big)^{\epsilon} \Big[ (\cfrac{\D\kappa}{\D\tau} \circ{\ucar}^{-1})^{\epsilon} + \big[\big( r(-\cfrac{t}{\epsilon})\,\mathbf{v}\big) \times \mathcal{M} \big] \cdot \big((\Nabla_{\mathbf{v}}\kappa) \circ{\ucar}^{-1}\big)^{\epsilon} \Big] \, d\mathbf{x}\,d\mathbf{v} \, dt \, .
\end{split}
\end{equation}
The third term gives
\begin{equation}
\begin{split}
\epsilon\, &\int_{0}^{T} \int_{\R^{6}} (l)^{\epsilon} \, \Big[ \mathbf{v} \cdot (\Nabla_{\mathbf{x}}\psi)^{\epsilon} + (\mathbf{E}+\mathbf{v} \times \mathbf{B}) \cdot (\Nabla_{\mathbf{v}}\psi)^{\epsilon} \Big] \, d\mathbf{x} \, d\mathbf{v}\, dt \\
&= \epsilon \, \int_{0}^{T} \int_{\R^{6}} (l)^{\epsilon} \, \Big[ \mathbf{v} \cdot (\cfrac{\D\Nabla_{\mathbf{x}}\kappa}{\D\tau})^{\epsilon} + (\mathbf{v} \times \mathcal{M})\cdot \big((\Nabla_{\mathbf{x}}\Nabla_{\mathbf{v}}\kappa)^{\epsilon}\,\mathbf{v}\big) + (\mathbf{E}+\mathbf{v} \times \mathbf{B}) \cdot (\cfrac{\D\Nabla_{\mathbf{v}}\kappa}{\D\tau})^{\epsilon} \\
&\qquad \qquad + (\mathbf{E}\times\mathcal{M} + \mathbf{v} \times \mathbf{B} \times \mathcal{M}) \cdot (\Nabla_{\mathbf{v}}\kappa)^{\epsilon} + (\mathbf{E}+\mathbf{v} \times \mathbf{B}) \cdot \big((\Nabla_{\mathbf{v}}^{\;2}\kappa)^{\epsilon}\,(\mathbf{v}\times\mathcal{M})\big) \Big] \, d\mathbf{x} \, d\mathbf{v} \, dt \, .
\end{split}
\end{equation}
Concerning the last term of (\ref{Proj_Ker_TS_2}), we have
\begin{equation}
\begin{split}
\int_{0}^{T} \int_{\R^{6}} &(h^{\epsilon})^{\epsilon} \, \big( (\cfrac{\D\psi}{\D t})^{\epsilon} + \cfrac{1}{\epsilon}\, (\cfrac{\D\psi}{\D \tau})^{\epsilon} \big) \, d\mathbf{x}\,d\mathbf{v} \, dt \\
&= \int_{0}^{T} \int_{\R^{6}} (h^{\epsilon})^{\epsilon} \, \cfrac{\D (\psi)^{\epsilon}}{\D t} \, d\mathbf{x}\,d\mathbf{v} \, dt \\
&= - \int_{0}^{T} \int_{\R^{6}} \cfrac{\D (h^{\epsilon})^{\epsilon}}{\D t} \, (\psi)^{\epsilon} \, d\mathbf{x}\,d\mathbf{v} \, dt + \int_{\R^{6}} h^{\epsilon}(0,0,\mathbf{x},\mathbf{v}) \, \psi(0,0,\mathbf{x},\mathbf{v}) \, d\mathbf{x}\, d\mathbf{v} \\
&= - \int_{0}^{T} \int_{\R^{6}} \big( (\cfrac{\D h^{\epsilon}}{\D t})^{\epsilon} + \cfrac{1}{\epsilon}\, (\cfrac{\D h^{\epsilon}}{\D \tau})^{\epsilon} \big) \, \big[ (\cfrac{\D\kappa}{\D\tau})^{\epsilon} + (\mathbf{v} \times \mathcal{M}) \cdot (\Nabla_{\mathbf{v}}\kappa)^{\epsilon} \big] \, d\mathbf{x} \, d\mathbf{v} \, dt \\
&\qquad \qquad + \int_{\R^{6}} h^{\epsilon}(0,0,\mathbf{x},\mathbf{v}) \, \big[ \cfrac{\D\kappa}{\D\tau}(0,0,\mathbf{x},\mathbf{v}) + (\mathbf{v} \times \mathcal{M}) \cdot \Nabla_{\mathbf{v}}\kappa(0,0,\mathbf{x},\mathbf{v}) \big] \, d\mathbf{x} \, d\mathbf{v} \\
&= - \int_{0}^{T}\int_{\R^{6}} \Big( (\cfrac{\D^{2}k^{\epsilon}}{\D t \, \D \tau})^{\epsilon} + (\mathbf{v} \times \mathcal{M}) \cdot (\cfrac{\D \Nabla_{\mathbf{v}}\kappa}{\D t})^{\epsilon} + \cfrac{1}{\epsilon}\, (\cfrac{\D^{2}k^{\epsilon}}{\D\tau^{2}})^{\epsilon} + \cfrac{1}{\epsilon}\, (\mathbf{v} \times \mathcal{M}) \cdot (\cfrac{\D \Nabla_{\mathbf{v}}\kappa}{\D \tau})^{\epsilon} \Big)\, \\
&\qquad \qquad \qquad \qquad \qquad \qquad \qquad \qquad \qquad \qquad \times \big[ (\cfrac{\D\kappa}{\D\tau})^{\epsilon} + (\mathbf{v} \times \mathcal{M}) \cdot (\Nabla_{\mathbf{v}}\kappa)^{\epsilon} \big] \, d\mathbf{x} \, d\mathbf{v} \, dt \\
&\qquad \qquad + \int_{\R^{6}} \big[ \cfrac{\D k^{\epsilon}}{\D\tau}(0,0,\mathbf{x},\mathbf{v}) + (\mathbf{v} \times \mathcal{M}) \cdot \Nabla_{\mathbf{v}}k^{\epsilon}(0,0,\mathbf{x},\mathbf{v})\big] \, \\
&\qquad \qquad \qquad \qquad \qquad \qquad \qquad \times \big[ \cfrac{\D\kappa}{\D\tau}(0,0,\mathbf{x},\mathbf{v}) + (\mathbf{v} \times \mathcal{M}) \cdot \Nabla_{\mathbf{v}}\kappa(0,0,\mathbf{x},\mathbf{v}) \big] \, d\mathbf{x} \, d\mathbf{v} \, ,
\end{split}
\end{equation}
on the one hand, and
\begin{equation}
\begin{split}
\int_{0}^{T}&\int_{\R^{6}} (h^{\epsilon})^{\epsilon}\,\Big[ \mathbf{v} \cdot (\Nabla_{\mathbf{x}}\psi)^{\epsilon} + \big( \mathbf{E} + \mathbf{v} \times (\mathbf{B} + \cfrac{\mathcal{M}}{\epsilon})\big) \cdot (\Nabla_{\mathbf{v}}\psi)^{\epsilon} \Big] \, d\mathbf{x}\,d\mathbf{v} \, dt \\
&= \int_{0}^{T} \int_{\R^{6}} \big( (\cfrac{\D k^{\epsilon}}{\D\tau})^{\epsilon} + (\mathbf{v} \times \mathcal{M}) \cdot (\Nabla_{\mathbf{v}}k^{\epsilon})^{\epsilon} \big) \, \Big[ \mathbf{v} \cdot (\cfrac{\D\Nabla_{\mathbf{x}}\kappa}{\D\tau})^{\epsilon} + (\mathbf{v} \times \mathcal{M}) \cdot \big( (\Nabla_{\mathbf{x}}\Nabla_{\mathbf{v}}\kappa)^{\epsilon}\,\mathbf{v}\big) \\
&\qquad \qquad \qquad + (\mathbf{E}+\mathbf{v}\times\mathbf{B}) \cdot (\cfrac{\D\Nabla_{\mathbf{v}}\kappa}{\D\tau})^{\epsilon} + (\mathbf{E} \times \mathcal{M} + (\mathbf{v} \times \mathbf{B}) \times \mathcal{M}) \cdot (\Nabla_{\mathbf{v}}\kappa)^{\epsilon} \\
&\qquad \qquad \qquad + (\mathbf{E}+\mathbf{v}\times\mathbf{B}) \cdot \big((\Nabla_{\mathbf{v}}^{\;2}\kappa)^{\epsilon}\,(\mathbf{v} \times \mathcal{M})\big) + \cfrac{1}{\epsilon}\,(\mathbf{v} \times \mathcal{M}) \cdot (\cfrac{\D\Nabla_{\mathbf{x}}\kappa}{\D\tau})^{\epsilon} \\
&\qquad \qquad \qquad + \big((\mathbf{v}\times\mathcal{M})\times\mathcal{M}\big) \cdot (\Nabla_{\mathbf{v}}\kappa)^{\epsilon} + (\mathbf{v} \times \mathcal{M}) \cdot \big( (\Nabla_{\mathbf{v}}^{\;2}\kappa)^{\epsilon}(\mathbf{v} \times \mathcal{M})\big) \Big] \, d\mathbf{x}\,d\mathbf{v}\,dt \, ,
\end{split}
\end{equation}
on the other hand. Remarking that
\begin{equation}
\big((\mathbf{v}\times\mathcal{M})\times\mathcal{M}\big) \cdot (\Nabla_{\mathbf{v}}\kappa)^{\epsilon} = -(\mathbf{v} \times \mathcal{M}) \cdot \big( (\Nabla_{\mathbf{v}}\kappa)^{\epsilon} \times \mathcal{M}\big) = -\mathbf{v} \cdot (\Nabla_{\mathbf{v}}\kappa)^{\epsilon} \, ,
\end{equation}
and dividing by $\epsilon$, we finally get the weak formulation of the Two-Scale Micro equation:
\begin{equation}\label{def_k_eps}
\begin{split}
&- \int_{0}^{T}\int_{\R^{6}} \Big( (\cfrac{\D^{2}k^{\epsilon}}{\D t \, \D \tau})^{\epsilon} + (\mathbf{v} \times \mathcal{M}) \cdot (\cfrac{\D \Nabla_{\mathbf{v}}k^{\epsilon}}{\D t})^{\epsilon} + \cfrac{1}{\epsilon}\, (\cfrac{\D^{2}k^{\epsilon}}{\D\tau^{2}})^{\epsilon} + \cfrac{1}{\epsilon}\, (\mathbf{v} \times \mathcal{M}) \cdot (\cfrac{\D \Nabla_{\mathbf{v}}k^{\epsilon}}{\D \tau})^{\epsilon} \Big)\, \\
&\qquad \qquad \qquad \qquad \qquad \qquad \qquad \qquad \qquad \qquad \qquad \big[ (\cfrac{\D\kappa}{\D\tau})^{\epsilon} + (\mathbf{v} \times \mathcal{M}) \cdot (\Nabla_{\mathbf{v}}\kappa)^{\epsilon} \big] \, d\mathbf{x} \, d\mathbf{v} \, dt \\
&\qquad \qquad + \int_{\R^{6}} \big[ \cfrac{\D k^{\epsilon}}{\D\tau}(0,0,\mathbf{x},\mathbf{v}) + (\mathbf{v} \times \mathcal{M}) \cdot \Nabla_{\mathbf{v}}k^{\epsilon}(0,0,\mathbf{x},\mathbf{v})\big] \, \\
&\qquad \qquad \qquad \qquad \qquad \qquad \qquad \qquad \big[ \cfrac{\D\kappa}{\D\tau}(0,0,\mathbf{x},\mathbf{v}) + (\mathbf{v} \times \mathcal{M}) \cdot \Nabla_{\mathbf{v}}\kappa(0,0,\mathbf{x},\mathbf{v}) \big] \, d\mathbf{x} \, d\mathbf{v} \\
&\qquad \qquad + \int_{0}^{T} \int_{\R^{6}} \big( (\cfrac{\D k^{\epsilon}}{\D\tau})^{\epsilon} + (\mathbf{v} \times \mathcal{M}) \cdot (\Nabla_{\mathbf{v}}k^{\epsilon})^{\epsilon} \big) \, \Big[ \mathbf{v} \cdot (\cfrac{\D\Nabla_{\mathbf{x}}\kappa}{\D\tau})^{\epsilon} + (\mathbf{v} \times \mathcal{M}) \cdot \big( (\Nabla_{\mathbf{x}}\Nabla_{\mathbf{v}}\kappa)^{\epsilon}\,\mathbf{v}\big) \\
&\qquad \qquad \qquad \qquad \qquad \qquad + (\mathbf{E}+\mathbf{v}\times\mathbf{B}) \cdot (\cfrac{\D\Nabla_{\mathbf{v}}\kappa}{\D\tau})^{\epsilon} + (\mathbf{E} \times \mathcal{M} + (\mathbf{v} \times \mathbf{B}) \times \mathcal{M}) \cdot (\Nabla_{\mathbf{v}}\kappa)^{\epsilon} \\
&\qquad \qquad \qquad \qquad \qquad \qquad + (\mathbf{E}+\mathbf{v}\times\mathbf{B}) \cdot \big((\Nabla_{\mathbf{v}}^{\;2}\kappa)^{\epsilon}\,(\mathbf{v} \times \mathcal{M})\big) + \cfrac{1}{\epsilon}\,(\mathbf{v} \times \mathcal{M}) \cdot (\cfrac{\D\Nabla_{\mathbf{x}}\kappa}{\D\tau})^{\epsilon} \\
&\qquad \qquad \qquad \qquad \qquad \qquad -\mathbf{v} \cdot (\Nabla_{\mathbf{v}}\kappa)^{\epsilon} + (\mathbf{v} \times \mathcal{M}) \cdot \big( (\Nabla_{\mathbf{v}}^{\;2}\kappa)^{\epsilon}(\mathbf{v} \times \mathcal{M})\big) \Big] \, d\mathbf{x}\,d\mathbf{v}\,dt \\
& \qquad \qquad - \int_{0}^{T} \int_{\R^{6}} \big(\cfrac{\D (l\circ{\ucar}^{-1})}{\D t}\big)^{\epsilon} \Big[ (\cfrac{\D\kappa}{\D\tau} \circ{\ucar}^{-1})^{\epsilon} + \big[\big( r(-\cfrac{t}{\epsilon})\,\mathbf{v}\big) \times \mathcal{M} \big] \cdot \big((\Nabla_{\mathbf{v}}\kappa) \circ{\ucar}^{-1}\big)^{\epsilon} \Big] \, d\mathbf{x}\,d\mathbf{v} \, dt \\
&\qquad \qquad + \int_{0}^{T} \int_{\R^{6}} (l)^{\epsilon} \, \Big[ \mathbf{v} \cdot (\cfrac{\D\Nabla_{\mathbf{x}}\kappa}{\D\tau})^{\epsilon} + (\mathbf{v} \times \mathcal{M})\cdot \big((\Nabla_{\mathbf{x}}\Nabla_{\mathbf{v}}\kappa)^{\epsilon}\,\mathbf{v}\big) + (\mathbf{E}+\mathbf{v} \times \mathbf{B}) \cdot (\cfrac{\D\Nabla_{\mathbf{v}}\kappa}{\D\tau})^{\epsilon} \\
&\qquad \qquad \qquad + (\mathbf{E}\times\mathcal{M} + (\mathbf{v} \times \mathbf{B}) \times \mathcal{M}) \cdot (\Nabla_{\mathbf{v}}\kappa)^{\epsilon} + (\mathbf{E}+\mathbf{v} \times \mathbf{B}) \cdot \big((\Nabla_{\mathbf{v}}^{\;2}\kappa)^{\epsilon}\,(\mathbf{v}\times\mathcal{M})\big) \Big] \, d\mathbf{x} \, d\mathbf{v} \, dt \\
&\qquad \qquad - \int_{0}^{T} \int_{\R^{6}} (\cfrac{\D F_{1}^{\epsilon}}{\D t})^{\epsilon} \Big[ (\cfrac{\D\kappa}{\D\tau})^{\epsilon} + (\mathbf{v} \times \mathcal{M}) \cdot (\Nabla_{\mathbf{v}}\kappa)^{\epsilon} \Big] \, d\mathbf{x} \, d\mathbf{v} \, dt \\
&\qquad \qquad + \int_{\R^{6}} F_{1}^{\epsilon}(0,0,\mathbf{x},\mathbf{v})\, \Big[ \cfrac{\D\kappa}{\D\tau}(0,0,\mathbf{x},\mathbf{v}) + (\mathbf{v} \times \mathcal{M}) \cdot \Nabla_{\mathbf{v}}\kappa(0,0,\mathbf{x},\mathbf{v})\Big] \, d\mathbf{x}\,d\mathbf{v} \\
&\qquad \qquad + \int_{0}^{T} \int_{\R^{6}} (F_{1}^{\epsilon})^{\epsilon} \, \Bigg[ \mathbf{v} \cdot (\cfrac{\D \Nabla_{\mathbf{x}}\kappa}{\D \tau})^{\epsilon} + (\mathbf{v} \times \mathcal{M}) \cdot \big((\Nabla_{\mathbf{x}}\Nabla_{\mathbf{v}}\kappa)^{\epsilon}\, \mathbf{v}\big) \\
&\qquad \qquad \qquad \qquad \qquad \qquad \qquad + (\mathbf{E} + \mathbf{v} \times \mathbf{B}) \, (\cfrac{\D \Nabla_{\mathbf{v}}\kappa}{\D \tau})^{\epsilon} + (\mathbf{E} \times \mathcal{M} + (\mathbf{v} \times \mathbf{B}) \times \mathcal{M}) \cdot (\Nabla_{\mathbf{v}}\kappa)^{\epsilon} \\
&\qquad \qquad \qquad \qquad \qquad \qquad \qquad + (\mathbf{E}+\mathbf{v}\times\mathcal{M}) \cdot \big((\Nabla_{\mathbf{v}}^{\;2}\kappa)^{\epsilon}(\mathbf{v}\times\mathcal{M})\big) \Bigg] \, d\mathbf{x}\,d\mathbf{v} \, dt = 0 \, .
\end{split}
\end{equation}
\begin{remark}\label{rem}
First, system \eqref{Proj_Ker_TS_final}-\eqref{def_k_eps} clearly has solutions,
since functions $G_{1}^{\epsilon}$ and $k^{\epsilon}$ as introduced in \eqref{decompo_feps_TS}
by using \eqref{decompo_orthogonal} were shown to satisfy
\eqref{Proj_Ker_TS_final}-\eqref{def_k_eps}. Then, we are aware that the Two-Scale
Macro-Micro system, \eqref{Proj_Ker_TS_final}-\eqref{def_k_eps}, has not necessarily a
unique solution.
The reason for the uniqueness lack is the replacement of $\tau$ by $\cfrac{t}{\epsilon}$.
In future work we will look for an additional condition for $G^{\epsilon}_1$ and
$k^{\epsilon}$ leading to the uniqueness of the whole solution
$(G^{\epsilon}_1,k^{\epsilon})$. Such a condition will be useful for the conception
of the numerical scheme for the Two-Scale Macro-Micro model in order to be sure that
when $\epsilon\sim 1$ we will approximate the solution of the starting Vlasov
equation \eqref{Vlasov_rescaled_intro}.
\end{remark}
\begin{remark}
Problem \eqref{eq_G}-\eqref{def_l}-\eqref{Proj_Ker_TS_final}-\eqref{def_k_eps} plays
the role of problem $\mathcal{E}^{\epsilon} \, U^{\epsilon} = 0$ in
diagram \eqref{diagram_TSAPS}.
\end{remark}

\section{Study of the asymptotic behaviour}
\label{SecRes} 
\setcounter{equation}{0}

\subsection{The Two-Scale Macro-Micro problem}
The Two-Scale Macro-Micro problem, \textit{i.e.} the system of equations
\eqref{eq_G}-\eqref{def_l}-\eqref{Proj_Ker_TS_final}-\eqref{def_k_eps}, is more
complicated to solve than the original problem \eqref{Vlasov_rescaled_intro},
but has two advantages over the latter. First, the Two-Scale
Macro-Micro model trivially reduces to the limit model \eqref{eq_G} by taking $\epsilon \to 0$ in \eqref{eq_G}-\eqref{def_l}-\eqref{Proj_Ker_TS_final}-\eqref{def_k_eps}. Second, the Two-Scale
Macro-Micro model decomposes the original solution $f^{\epsilon}$ into a Macro part,
$G\circ\ucar + \epsilon\,G_{1}^{\epsilon}\circ\ucar$, and a Micro part,
$\epsilon\,l + \epsilon\,h_{\epsilon}$. This question is relevant when doing the
transition from the very small $\epsilon$ regime to the one of $\epsilon\sim 1$,
since our model describes separately the evolution at the macroscopic time scale
(of $G\circ\ucar$ and $G_{1}^{\epsilon}\circ\ucar$) which contains essential oscillation through
$\ucar$, and the evolution of oscillation corrections ($l$ and $h_{\epsilon}$) at
the microscopic time scale.

Now it is easy to see that, under the hypothesis of uniqueness of the solution of
equations \eqref{Proj_Ker_TS_final} and \eqref{def_k_eps} (see Remark~\ref{rem}),
the Two-Scale Macro-Micro problem
\eqref{eq_G}-\eqref{def_l}-\eqref{Proj_Ker_TS_final}-\eqref{def_k_eps} is equivalent
to the original problem \eqref{Vlasov_rescaled_intro}. Indeed, let $f^{\epsilon}$ be the
solution of \eqref{Vlasov_rescaled_intro}. We have seen, by using Lemma~\ref{lemma:2},
that the decomposition \eqref{decompo_feps_TS} exists and is unique. Then, by the
calculations we did in the previous sections, we obtain that $G$ is solution to
\eqref{eq_G}, $l$ is given by the formula \eqref{def_l}, $G_{1}^{\epsilon}$ is
solution to \eqref{Proj_Ker_TS_final} and $k^{\epsilon}$ to \eqref{def_k_eps}.
In particular, we have proved that the equation system
\eqref{Proj_Ker_TS_final}-\eqref{def_k_eps} has solution.

Conversely, assume that the solution of the system
\eqref{Proj_Ker_TS_final}-\eqref{def_k_eps} is unique. Let
$(G,l,G_{1}^{\epsilon},k^{\epsilon})$ be the solution of
\eqref{eq_G}-\eqref{def_l}-\eqref{Proj_Ker_TS_final}-\eqref{def_k_eps}. Then constructing
$f^{\epsilon}$ by \eqref{moregeneral_decompo_feps_TS}, we will obtain the solution
of \eqref{Vlasov_rescaled_intro}, since this problem has unique solution.

\subsection{Macro-Micro decomposition \emph{vs} Two-Scale Macro-Micro decomposition}
\label{SecTSMMvsMM}

As it has been mentioned in the previous sections, the Two-Scale Macro-Micro decomposition model \eqref{eq_G}-\eqref{def_l}-\eqref{Proj_Ker_TS_final}-\eqref{def_k_eps} is linked to the usual Macro-Micro model \eqref{eq_m}, \eqref{proj_Ker_weak}, \eqref{proj_Im_weak}, \eqref{proj_Ker_weak_init},
\eqref{proj_Im_weak_init}.

\begin{property}
Integrating in $\tau$ the Two-Scale Macro-Micro decomposition leads to the classical
Macro-Micro model developed in Section~\ref{Sec:Mmdeco}. 
\end{property}
Indeed, let us recall
that the unique decomposition made in \eqref{decompo_feps_TS} is of the kind
\begin{equation}\label{eco:1}
f^{\epsilon}(t,\mathbf{x},\mathbf{v}) = F(t,\tau,\mathbf{x},\mathbf{v})\,+\,
F_1^{\epsilon}(t,\tau,\mathbf{x},\mathbf{v})+H_1^{\epsilon}(t,\tau,\mathbf{x},\mathbf{v}),
\end{equation}
for some $F_1^{\epsilon}\in\textnormal{Ker}\Big(\cfrac{\D}{\D\tau} + (\mathbf{v}\times
\mathcal{M})\cdot\Nabla_{\mathbf{v}}\Big)$ and $H_1^{\epsilon}\in\textnormal{Im}\Big(
\cfrac{\D}{\D\tau} + (\mathbf{v} \times \mathcal{M}) \cdot \Nabla_{\mathbf{v}}\Big)$.
Then, integrating \eqref{eco:1} in $\tau$, we obtain using \eqref{link_fF}
\begin{equation}\label{eco:2}
f^{\epsilon}(t,\mathbf{x},\mathbf{v})=f(t,x,\mathbf{v})\,+\int_0^{1}
F_1^{\epsilon}(t,\tau,\mathbf{x},\mathbf{v})\,d\tau \,+\int_0^{1}
H_1^{\epsilon}(t,\tau,\mathbf{x},\mathbf{v})\,d\tau.
\end{equation}
Now, changing the variable $\mathbf{v}$ into $(v_{||},v_{\perp},\alpha)$ defined in
\eqref{cyl:coord}, and using Lemma~\ref{lemma:3}, we obtain that the function
$(v_{\perp},\alpha) \mapsto \displaystyle{\int_0^{1}F_1^{\epsilon}\big(t,\tau,\mathbf{x},
v_{||}\,\mathbf{e}_{1}+v_{\perp}(\cos\alpha\,\mathbf{e}_{2}+\sin\alpha\,
\mathbf{e}_{3})\big)\,d\tau}$ is in $\textnormal{Ker}\Big(\cfrac{\D}{\D\alpha}\Big)$.
Similarly, \hspace{1mm}  $(v_{\perp},\alpha) \mapsto \displaystyle{\int_0^{1}
H_1^{\epsilon}\big(t,\tau,\mathbf{x},v_{||}\,\mathbf{e}_{1}+v_{\perp}(\cos\alpha\,
\mathbf{e}_{2}+\sin\alpha\,\mathbf{e}_{3})\big)\,d\tau}$ is in $\textnormal{Im}\Big(
\cfrac{\D}{\D\alpha}\Big)$. Next, writing \eqref{eco:2} in the new variables and recalling
the decomposition of $f^{\epsilon}(t,\mathbf{x},\mathbf{v})$ in \eqref{decompo_feps_weak},
we deduce from \eqref{link:mf} and the uniqueness of such a decomposition (stated in
Lemma~\ref{lemma:1}(iii)) that
\begin{equation}
\begin{split}
m_1^{\epsilon}(t,\mathbf{x},v_{||},v_{\perp}) & = \int_0^{1}F_1^{\epsilon}\big(t,\tau,\mathbf{x},
v_{||}\,\mathbf{e}_{1}+v_{\perp}(\cos\alpha\,\mathbf{e}_{2}+\sin\alpha\,
\mathbf{e}_{3})\big)\,d\tau, \\
n^{\epsilon}(t,\mathbf{x},v_{||},v_{\perp},\alpha) & = \int_0^{1}
H_1^{\epsilon}\big(t,\tau,\mathbf{x},v_{||}\,\mathbf{e}_{1}+v_{\perp}(\cos\alpha\,
\mathbf{e}_{2}+\sin\alpha\,\mathbf{e}_{3})\big)\,d\tau.
\end{split}
\end{equation}

\subsection{Convergence of the Two-Scale Macro-Micro problem}
\indent Let us recall that our Two-Scale Macro-Micro decomposition \eqref{decompo_feps_TS}
is based on the convergence result in \eqref{1stord_conv}. More precisely, keeping in mind
that $G+\epsilon\,\widetilde{G}+\epsilon \,l$ is the first order approximation of
$f^{\epsilon}$, \eqref{decompo_feps_TS} can be seen as a Macro-Micro decomposition of
$f^{\epsilon}$ at the first order level of approximation. Next, we justify this approximation
using the Two-Scale convergence.

\begin{theorem}\label{conv}
We assume that $f_{0} \in L^{2}(\R^{6})$, $\mathbf{E}\in W^{1,\infty}(\R^{3})$, 
$\mathbf{B}\in W^{1,\infty}(\R^{3})$, and $\D l/\D t,\,\nabla_{\mathbf{x},\mathbf{v}} l\in
L^{\infty}\big(0,T;L^{\infty}(0,1;L^2(\R^{6}))\big)$. Then, when $\epsilon\to0$, the solutions
$G^{\epsilon}_1\circ\ucar$ of \eqref{Proj_Ker_TS_final} Two-Scale converge to
$\widetilde{G}\in L^{\infty}\big(0,T;L^2(\R^{6})\big)$, the solution of \eqref{G_tilde}.
When $\epsilon\to0$, the solutions $k^{\epsilon}$ of \eqref{def_k_eps} Two-Scale converge to $0$.
\end{theorem}
\begin{proof}
Recall from Theorem 1.5 of \cite{Raviart} that when $\epsilon\to0$, $\Big(\frac{1}{\epsilon}
\big(f^{\epsilon}-G\circ\ucar\big) - l\Big)_{\epsilon>0}$ Two-Scale converges to
$\widetilde{G}\circ\ucar\in\textnormal{Ker}\Big(\cfrac{\D}{\D\tau} + (\mathbf{v} \times \mathcal{M})
\cdot \Nabla_{\mathbf{v}} \Big)$.
Then, since \eqref{moregeneral_decompo_feps_TS} implies that $$\frac{1}{\epsilon}
\big(f^{\epsilon}-G\circ\ucar\big)-l = G_1^{\epsilon}\circ\ucar + h^{\epsilon},$$ and since
$G_1^{\epsilon}\circ\ucar\in\textnormal{Ker}\Big(\cfrac{\D}{\D\tau} + (\mathbf{v} \times
\mathcal{M}) \cdot \Nabla_{\mathbf{v}} \Big)$ and
$h^{\epsilon}\in\textnormal{Im}\Big(\cfrac{\D}{\D\tau} + (\mathbf{v} \times \mathcal{M})
\cdot \Nabla_{\mathbf{v}} \Big)$, the theorem's conclusion is clearly true.
\end{proof}

\begin{remark}
Thanks to the previous results, the Two-Scale Macro-Micro model we have built satisfies the expected asymptotic behaviour symbolized by the arrows in the top layer of diagram \eqref{diagram_TSAPS}.
Concerning the time discretization of
\eqref{eq_G}-\eqref{def_l}-\eqref{Proj_Ker_TS_final}-\eqref{def_k_eps}, we intend
to proceed as discussed in Remark~\ref{rem:apCMM} for the classical Macro-Micro
formulation (see \cite{Bennoune, Crouseilles, Lemou-Mieussens}), in order to 
obtain Asymptotic-Preserving schemes, symbolized by the arrows in the bottom
layer of diagram \eqref{diagram_TSAPS}.

\end{remark}

\begin{appendices}

In Appendix A, we characterize the projection onto $\textnormal{Ker}
(\frac{\D}{\D\alpha})$. We establish that it consists in computing its average in
$\alpha$ and projecting it onto $\textnormal{Im}(\frac{\D}{\D\alpha})$ consists in
substracting from it its average value. Then, in the second Appendix, we do the same
with operator $\frac{\D}{\D\tau} + (\mathbf{v} \times \mathcal{M}) \cdot
\Nabla_{\mathbf{v}}$. We also establish the link that exists between these two operators.

\section{About the operator $\cfrac{\D}{\D\alpha}$}
\setcounter{equation}{0}

\begin{lemma}\label{lemma:1}
The unbounded operator $\cfrac{\D}{\D\alpha} : L_{\#_{2\pi}}^{2}(\R; L^{2}(\R^{+};
v_{\perp}\,dv_{\perp})) \to L_{\#_{2\pi}}^{2}(\R; L^{2}(\R^{+}; v_{\perp}\,dv_{\perp}))$
has a closed range in $L_{\#_{2\pi}}^{2}(\R; L^{2}(\R^{+}; v_{\perp}\,dv_{\perp}))$ and
\begin{itemize}
\item[(i)] $\textnormal{Ker}\Big(\cfrac{\D}{\D\alpha}\Big)^{\perp} = \textnormal{Im}
\Big(\cfrac{\D}{\D\alpha}\Big)$, \\ 
\item[(ii)]
For any function $f \in L_{\#_{2\pi}}^{2}(\R; L^{2}(\R^{+}; v_{\perp}\,dv_{\perp}))$, the
projection of $f$ on $\textnormal{Ker}\Big(\cfrac{\D}{\D\alpha}\Big)$ is the function
\begin{equation}\label{def:proj}
Pf\;:\;(v_{\perp},\alpha)\,\longmapsto\,\cfrac{1}{2\pi}\,\int_{0}^{2\pi} f(v_{\perp},
\theta)\,d\theta\,,
\end{equation}
\item[(iii)] $L_{\#_{2\pi}}^{2}(\R; L^{2}(\R^{+}; v_{\perp}\,dv_{\perp})) = \textnormal{Ker}
\Big(\cfrac{\D}{\D\alpha}\Big) \oplus \textnormal{Im}\Big(\cfrac{\D}{\D\alpha}\Big)$.
\end{itemize}
\end{lemma}

\begin{proof}
$(i)$ Let $q\in\Big( \textnormal{Ker}\big(\cfrac{\D}{\D\alpha}\big) \Big)^{\perp}$.
Then, for any $g \in \textnormal{Ker}\Big(\cfrac{\D}{\D\alpha}\Big)$,
\textit{i.e.} $g = g(v_{\perp})$, we have
\begin{equation}\label{A1}
\int_{0}^{2\pi}\int_{0}^{+\infty} q(v_{\perp},\alpha)\,g(v_{\perp}) \, v_{\perp}\,dv_{\perp}\,
d\alpha = 0 \, .
\end{equation}
Remarking that $q$ can be written as
\begin{equation}
q(v_{\perp},\alpha) = \cfrac{\D}{\D\alpha} \Bigg( \int_{0}^{\alpha} q(v_{\perp},\theta)\,
d\theta\Bigg) \, ,
\end{equation}
we only have to prove that the function $K$ defined by
\begin{equation}
K(v_{\perp},\alpha) = \int_{0}^{\alpha} q(v_{\perp},\theta)\,d\theta \, ,
\end{equation}
is in $L_{\#_{2\pi}}^{2}(\R;L^{2}(\R^{+}; v_{\perp}\,dv_{\perp}))$, \textit{i.e.} $K$ is
$2\pi$-periodic in $\alpha$ and $\displaystyle \int_{0}^{2\pi}\int_{0}^{+\infty} \big|
K(v_{\perp},\alpha)\big|^{2} \,v_{\perp}\,dv_{\perp}\,d\alpha < +\infty$. \\
\indent First, we have
\begin{equation}
K(v_{\perp},\alpha+2\pi)-K(v_{\perp},\alpha)=\int_{0}^{2\pi}q(v_{\perp},\theta)\,d\theta=0\,,
\end{equation}
by \eqref{A1}. Second, using Jensen's inequality, we remark that
\begin{equation}
\big|K(v_{\perp},\alpha)\big|^{2} \leq |\alpha| \,\int_{0}^{\alpha} \big|q(v_{\perp},\theta)
\big|^{2} \,d\theta \leq 2\pi\,\int_{0}^{2\pi}\big|q(v_{\perp},\theta)\big|^{2}\,d\theta\,.
\end{equation}
Then
\begin{equation}
\int_{0}^{2\pi} \int_{0}^{+\infty} \big|K(v_{\perp},\alpha)\big|^{2}\,v_{\perp}\,dv_{\perp}\,
d\alpha \leq 4\pi^{2}\, \int_{0}^{2\pi}\int_{0}^{+\infty}\big|q(v_{\perp},\theta)\big|^{2}\,
v_{\perp}\,dv_{\perp}\,d\theta < +\infty \, ,
\end{equation}
since $q \in L_{\#_{2\pi}}^{2}(\R; L^{2}(\R^{+}; v_{\perp}\,dv_{\perp}))$. Thus $q \in
\textnormal{Im}\Big(\cfrac{\D}{\D\alpha}\Big)$ and so $\textnormal{Ker}\Big(
\cfrac{\D}{\D\alpha}\Big)^{\perp} \subset \textnormal{Im}\Big(\cfrac{\D}{\D\alpha}\Big)$.
Since the converse inclusion is obvious, $(i)$ is proved and the range of the operator
is indeed closed.

$(ii)$ It is clear that $f\in L_{\#_{2\pi}}^{2}(\R; L^{2}(\R^{+}; v_{\perp}\,dv_{\perp}))$
implies that $Pf\in L_{\#_{2\pi}}^{2}(\R; L^{2}(\R^{+}; v_{\perp}\,dv_{\perp}))$. Then, using
the projection characterization, we have for any function $\psi \in L_{\#_{2\pi}}^{2}(\R;
L^{2}(\R^{+}; v_{\perp}\,dv_{\perp}))$ such that $\cfrac{\D\psi}{\D\alpha} = 0$,
\begin{equation}
\begin{split}
\int_{0}^{2\pi} &\int_{0}^{+\infty} \Bigg( f(v_{\perp},\alpha)-\cfrac{1}{2\pi}\,
\int_{0}^{2\pi}f(v_{\perp},\theta)\,d\theta\Bigg)\,\psi(v_{\perp})\,v_{\perp}\,dv_{\perp}\,
d\alpha \\
& =\int_{0}^{+\infty}\Bigg(\int_{0}^{2\pi}f(v_{\perp},\alpha)\,d\alpha\Bigg)\,\psi(v_{\perp})
\,v_{\perp}\,dv_{\perp} -\Bigg( \cfrac{1}{2\pi}\int_{0}^{2\pi}\,d\alpha\Bigg) \,
\int_{0}^{+\infty}\Bigg(\int_{0}^{2\pi}f(v_{\perp},\theta)\,d\theta\Bigg)\,\psi(v_{\perp})
\,v_{\perp}\,dv_{\perp} \\
& =0 \, .\nonumber
\end{split}
\end{equation}
Hence $Pf$ defined in \eqref{def:proj} is the projection of $f$ on $\textnormal{Ker}
\Big(\cfrac{\D}{\D\alpha}\Big)$.

$(iii)$ Using item $(i)$ we have only to prove that any $f \in L_{\#_{2\pi}}^{2}(\R;
L^{2}(\R^{+}; v_{\perp}\,dv_{\perp}))$ writes as a sum $f=g+h$ with $g \in \textnormal{Ker}\Big(
\cfrac{\D}{\D\alpha}\Big)$ and $h\in\textnormal{Im}\Big(\cfrac{\D}{\D\alpha}\Big)$. Given
$f\in L_{\#_{2\pi}}^{2}(\R; L^{2}(\R^{+}; v_{\perp}\,dv_{\perp}))$, let $g$ be the projection
of $f$ on $\textnormal{Ker}\Big(\cfrac{\D}{\D\alpha}\Big)$ and let $h=f-g$. Then, by item
$(ii)$ we only need to show that $h \in \textnormal{Im}\Big(\cfrac{\D}{\D\alpha}\Big)$.
This can be done using a similar way as in the proof of item $(i)$. 
This concludes the proof of the Lemma.
\end{proof}

\section{The operator $\cfrac{\D}{\D\tau}+(\mathbf{v}\times\mathcal{M})
\cdot\Nabla_{\mathbf{v}}$}
\begin{lemma}\label{lemma:2}
Let the unbounded operator \par 
$\cfrac{\D}{\D\tau}+(\mathbf{v} \times \mathcal{M})
\cdot \Nabla_{\mathbf{v}} \,:\,L^{\infty}\big(0,T;L_{\#_{1}}^{\infty}(\R^{+};L^{2}(\R^{6}))\big)
\,\to\, L^{\infty}\big(0,T;L_{\#_{1}}^{\infty}(\R^{+}; L^{2}(\R^{6}))\big)$. Then we have
\begin{equation} \label{deco_ortho}
L^{\infty}\big(
0,T; L_{\#_{1}}^{\infty}(\R^{+}; L^{2}(\R^{6}))\big) = \textnormal{Ker}\left(\cfrac{\D}{\D\tau}+
(\mathbf{v} \times \mathcal{M})\cdot \Nabla_{\mathbf{v}} \right) \oplus \textnormal{Im}
\Bigg(\cfrac{\D}{\D\tau}+(\mathbf{v} \times \mathcal{M}) \cdot \Nabla_{\mathbf{v}} \Bigg).
\end{equation}
\end{lemma}

\begin{proof}
Let $f \in L^{\infty}\big(0,T; L_{\#_{1}}^{\infty}(\R^{+}; L^{2}(\R^{6}))\big)$.
Through the change of variables
\begin{equation}
\left\{
\begin{array}{l}
\textnormal{$v_{||}\in\R$ such that $v_{||} = \mathbf{v} \cdot \mathbf{e}_{1} \,$,} \\
\textnormal{$v_{\perp}\in\R^+$ such that $v_{\perp} = \sqrt{v_{2}^{2}+v_{3}^{2}} \, $,} \\
\textnormal{$\alpha\in [0,2\pi]$ such that $v_{2} = v_{\perp}\,\cos\alpha$ and
$v_{3} = v_{\perp}\,\sin\alpha$,}
\end{array}
\right.
\end{equation}
$\widetilde f$ \Big(where $\widetilde f(t,\tau,\mathbf{x},v_{||},v_{\perp},\alpha)\,=\,
f(t,\tau,\mathbf{x},v_{||}\,\mathbf{e}_{1}+v_{\perp}(\cos\alpha\,\mathbf{e}_{2}+\sin\alpha
\,\mathbf{e}_{3}))$\Big) is in $L^{\infty}\big(0,T; L_{\#_{1}}^{\infty}(\R^{+}; L_{\#_{2\pi}}^{2}(
\R;$ $L^{2}(\R^{3}\times\R\times \R^{+}; v_{\perp}\,d\mathbf{x}\,dv_{||}\,dv_{\perp})))
\big)$ while the operator $\cfrac{\D}{\D\tau}+(\mathbf{v} \times \mathcal{M})\cdot
\Nabla_{\mathbf{v}}$ writes $\cfrac{\D}{\D\tau}-2\pi\cfrac{\D}{\D\alpha}\,$. Next,
performing the second change of variables
\begin{equation}
\left\{
\begin{array}{l}
\sigma = \tau + \cfrac{1}{2\pi}\:\alpha \, ,\\
\beta = \tau - \cfrac{1}{2\pi}\:\alpha \, ,
\end{array}
\right.
\end{equation}
$\Check f$ \Big(where $\Check f(t,\sigma,\mathbf{x},v_{||},v_{\perp},\beta)\,=\,
\widetilde f(t,(\sigma+\beta)/2,\mathbf{x},v_{||},v_{\perp},\pi(\sigma-\beta))$\Big)
\, stays in \, $L^{\infty}\big(0,T; L_{\#_{1}}^{\infty}(\R^{+}; L_{\#_{2\pi}}^{2}(\R;$ 
$L^{2}(\R^{3}\times\R\times \R^{+}; v_{\perp}\,d\mathbf{x}\,dv_{||}\,dv_{\perp})))
\big)$ and the operator
$\cfrac{\D}{\D\tau}-2\pi\cfrac{\D}{\D\alpha}$ \big(and consequently $\cfrac{\D}{\D\tau}+
(\mathbf{v} \times \mathcal{M}) \cdot \Nabla_{\mathbf{v}}$\big) becomes $2\cfrac{\D}{\D
\beta}\,$. From now on, let us fix $(t,\mathbf{x})\in[0,T]\times\R^{3}$. Considering 
\begin{equation}
\begin{split}
\frac{\D}{\D\beta}: L_{\#_{1}}^{\infty}(\R^{+}; L_{\#_{2\pi}}^{2}(\R;\, 
& L^{2}(\R\times \R^{+}; v_{\perp}\,dv_{||}\,dv_{\perp}))) \\
&\to L_{\#_{1}}^{\infty}(\R^{+}; L_{\#_{2\pi}}^{2}(
\R; L^{2}(\R\times \R^{+}; v_{\perp}\,dv_{||}\,dv_{\perp}))),
\end{split}
\end{equation}
we know, by Lemma~\ref{lemma:1}, that $L_{\#_{2\pi}}^{2}(\R;L^{2}(\R^{+}; v_{\perp}\,
dv_{\perp})) = \textnormal{Ker}\Big(\cfrac{\D}{\D\beta}\Big) \oplus \textnormal{Im}
\Big(\cfrac{\D}{\D\beta}\Big)$. Hence, because $L_{\#_{1}}^{\infty}(\R^{+};$ 
$L_{\#_{2\pi}}^{2}(\R;L^{2}(\R\times\R^{+}; v_{\perp}\,dv_{||}\,dv_{\perp}))) \subset
L_{\#_{1}}^{2}(\R^{+}; L_{\#_{2\pi}}^{2}(\R;L^{2}(\R\times\R^{+}; v_{\perp}\,dv_{||}\,
dv_{\perp})))$, we can write $\check f=P\check f+(\check f-P\check f)$,
where $P\check f$ is defined as in \eqref{def:proj}. Since $\check f$ was chosen in
$L_{\#_{1}}^{\infty}(\R^{+}; L_{\#_{2\pi}}^{2}(\R;L^{2}(\R\times\R^{+}; v_{\perp}\,dv_{||}\,
dv_{\perp})))$, it is easy to see that $P\check f$ belongs to this same space. Obviously
$\check f-P\check f\in L_{\#_{1}}^{\infty}(\R^{+};L_{\#_{2\pi}}^{2}(\R;L^{2}(\R \times\R^{+};
v_{\perp}\,dv_{||}\,dv_{\perp})))$ and thus \eqref{deco_ortho} is proved.
\end{proof}

\begin{lemma}\label{lemma:3}
Let $f \in L^{\infty}\big(0,T; L_{\#_{1}}^{\infty}(\R^{+}; L^{2}(\R^{6}))\big)$ and let
$\widetilde f$ be defined as in the proof of Lemma~\ref{lemma:2}. Then
\begin{itemize}
\item[(i)] if $f\in\textnormal{Ker}\Big(\cfrac{\D}{\D\tau} + (\mathbf{v}\times\mathcal{M})
\cdot \Nabla_{\mathbf{v}} \Big)$ then for any fixed $(t,\mathbf{x},v_{||})\in [0,T]\times
\R^3\times\R$ we have\vspace*{-3mm}
$$\mathcal{F}\,:\,(v_{\perp},\alpha) \mapsto \displaystyle{\int_0^{1}\widetilde f(t,\tau,
\mathbf{x},v_{||},v_{\perp},\alpha)\,d\tau}\in\textnormal{Ker}\Big(\cfrac{\D}{\D\alpha}\Big)\, ,$$
\vspace*{-9mm}
\\
\item[(ii)] if $f\in\textnormal{Im}\Big(\cfrac{\D}{\D\tau} + (\mathbf{v}\times\mathcal{M})
\cdot \Nabla_{\mathbf{v}} \Big)$ then for any fixed $(t,\mathbf{x},v_{||})\in [0,T]\times
\R^3\times\R$ we have\vspace*{-3mm}
$$\mathcal{F}\,:\,(v_{\perp},\alpha) \mapsto \displaystyle{\int_0^{1}\widetilde f(t,\tau,
\mathbf{x},v_{||},v_{\perp},\alpha)\,d\tau}\in\textnormal{Im}\Big(\cfrac{\D}{\D\alpha}\Big).$$
\end{itemize}
\end{lemma}
\begin{proof}
$(i)$ Let $f\in\textnormal{Ker}\Big(\cfrac{\D}{\D\tau}+(\mathbf{v}\times\mathcal{M})\cdot
\Nabla_{\mathbf{v}} \Big)$ and fix $(t,\mathbf{x},v_{||})\in [0,T]\times\R^3\times\R$.
Thus, by the proof of Lemma~\ref{lemma:2}, $\widetilde f\in\textnormal{Ker}\Big(
\cfrac{\D}{\D\tau}-2\pi\cfrac{\D}{\D\alpha}\Big)$. We then derive $\mathcal{F}$ :
\begin{eqnarray}
\cfrac{\D\mathcal{F}}{\D\alpha}\,(v_{\perp},\alpha) & = & \int_0^{1}\cfrac{\D\widetilde f}
{\D\alpha}(t,\tau,\mathbf{x},v_{||},v_{\perp},\alpha) \,d\tau \nonumber \\
& = & \cfrac{1}{2\pi}\int_0^{1}\cfrac{\D\widetilde f}{\D\tau}
(t,\tau,\mathbf{x},v_{||},v_{\perp},\alpha) \,d\tau \nonumber \\
& = & 0\,,
\end{eqnarray}
by the $1$-periodicity in $\tau$ of $\widetilde f$. Thus, $\mathcal{F}\in\textnormal{Ker}
\Big(\cfrac{\D}{\D\alpha}\Big)$. \\
$(ii)$ Now let $f\in\textnormal{Im}\Big(\cfrac{\D}{\D\tau} 
+(\mathbf{v}\times\mathcal{M})\cdot\Nabla_{\mathbf{v}} \Big)$. Then there exists a function
$h \in L^{\infty}\big(0,T; L_{\#_{1}}^{\infty}(\R^{+}; L^{2}(\R^{6}))\big)$ such that $\widetilde f
= \cfrac{\D \widetilde h}{\D\tau}-2\pi\cfrac{\D\widetilde h}{\D\alpha}$\,. Fixing $(t,\mathbf{x},
v_{||}) \in [0,T]\times\R^3\times\R$ and integrating in $\tau$ from $0$ to $1$,
we obtain, using that $h$ is $1$-periodic in $\tau$,
$$\mathcal{F}(v_{\perp},v_{||})= - 2\pi\cfrac{\D}{\D\alpha}\int_0^{1} \widetilde h(t,\tau,
\mathbf{x},v_{||},v_{\perp},\alpha) \,d\tau,$$ leading to $\mathcal{F}\in
\textnormal{Im}\Big(\cfrac{\D}{\D\alpha}\Big)$ and thus concluding the Lemma.
\end{proof}

\end{appendices}


\begin{thebibliography}{1}
\bibitem{Ailliot-Frenod-Monbet} \textsc{Ailliot, P., Fr\'enod, E., Monbet, V.}, \emph{Long term object drift in the ocean with tide and wind}, Multiscale Model. Simul. \textbf{5}-2 (2006), 514-531.

\bibitem{Allaire} \textsc{Allaire, G.}, \emph{Homogenization and two-scale convergence}, SIAM J. Math. Anal. \textbf{23}-6 (1992), 1482-1518.

\bibitem{Bennoune} \textsc{Bennoune, M., Lemou, M., Mieussens, L.}, \emph{Uniformly stable
numerical schemes for the Boltzmann equation preserving the compressible Navier-Stokes
asymptotics}, J. Comput. Phys. \textbf{227}-8 (2008), 3781-3803.

\bibitem{Bostan2007} \textsc{Bostan, M.}, \emph{The Vlasov-Poisson system with strong external magnetic field. Finite Larmor radius regime}, Asymptot. Anal. \textbf{61}-2 (2009), 91-123.

 \bibitem{Bostan_2010} \textsc{Bostan, M.}, \emph{Transport equations with disparate advection fields. Application to the gyrokinetic models in plasma physics}, 
J. Diff. Eq., {\bf 249} (2010), 1620-1663. 

\bibitem{Crouseilles} \textsc{Crouseilles, N., Lemou, M.}, \emph{An asymptotic preserving
scheme based on a micro-macro decomposition for Collisional Vlasov equations: diffusion and
high-field scaling limits}, Kinet. Relat. Models \textbf{4}-2 (2011), 441-477.

\bibitem{filbet-jin}
\textsc{Filbet, F., Jin, S.}, 
\emph{A class of asymptotic preserving schemes for kinetic equations and related problems with stiff sources}, 
J. Comp. Phys. \textbf{229}-20 (2010), 7625-7648. 

\bibitem{Frenod-Mouton} \textsc{Fr\'enod, E., Mouton, A.}, \emph{Two-dimensional Finite Larmor Radius approximation in canonical gyrokinetic coordinates}, J. Pure Appl. Math. Adv. Appl. \textbf{4}-2 (2010), 135-166.

\bibitem{Frenod-Mouton-Sonnen} \textsc{Fr\'enod, E., Mouton, A., Sonnendr\"ucker, E.}, \emph{Two-scale numerical simulation of the weakly compressible 1D isentropic Euler equations}, Numer. Math. \textbf{108}-2 (2007), 263-293.

\bibitem{Raviart} \textsc{Fr\'enod, E., Raviart, P.-A., Sonnendr\"ucker, E.}, \emph{Two-scale expansion of a singularly perturbed convection equation}, J. Math. Pures Appl. \textbf{80}-8 (2001), 815-843.

\bibitem{Frenod-Salvarani-Sonnen} \textsc{Fr\'enod, E., Salvarani, F., Sonnendr\"ucker, E.}, \emph{Long time simulation of a beam in a periodic focusing channel via a two-scale PIC-method}, Math. Models Methods Appl. Sci. \textbf{19}-2 (2009), 175-197.

\bibitem{Frenod-Sonnen_Homogenization} \textsc{Fr\'enod, E., Sonnendr\"ucker, E.}, \emph{Homogenization of the Vlasov equation and of the Vlasov-Poisson system with a strong external magnetic field}, Asymptot. Anal. \textbf{18}-3-4 (1998), 193-214.

\bibitem{Frenod-Sonnen_FLR} \textsc{Fr\'enod, E., Sonnendr\"ucker, E.}, \emph{The Finite Larmor Radius approximation}, SIAM J. Math. Anal. \textbf{32}-6 (2001), 1227-1247.

\bibitem{Golse-Saint-Raymond_1999} \textsc{Golse, F., Saint-Raymond, L.}, \emph{The Vlasov-Poisson system with strong magnetic field}, J. Math. Pures Appl. \textbf{78} (1999), 791-817.

\bibitem{Golse-Saint-Raymond_2003} \textsc{Golse, F., Saint-Raymond, L.}, \emph{The Vlasov-Poisson system with strong magnetic field in quasineutral regime}, Math. Models Methods Appl. Sci. \textbf{13}-5 (2003), 661-714.

\bibitem{Han-Kwan} \textsc{Han-Kwan, D.}, \emph{The three-dimensional finite Larmor radius approximation}, Asymptot. Anal. \textbf{66}-1 (2010), 9-33.

\bibitem{Jin} \textsc{Jin, S.}, \emph{Efficient asymptotic-preserving (AP) schemes for some
multiscale kinetic equations}, SIAM J. Sci. Comput. \textbf{21}-2 (1999), 441-454.

\bibitem{Jin-Shi} \textsc{Jin, S., Shi, Y.}, \emph{A Micro-Macro Decomposition-Based
Asymptotic-Preserving Scheme for the Multispecies Boltzmann Equation}, SIAM J. Sci. Comp.
\textbf{31}-6 (2010), 4580-4606.

 \bibitem{klar} \textsc{Klar, A.}, 
\emph{A Numerical Method for Kinetic Semiconductor Equations in the Drift Diffusion Limit},
SIAM J. Sci. Comp. {\bf 20}-5 (1999), 1696-1712. 

\bibitem{Lemou-Mieussens} \textsc{Lemou, M., Mieussens, L.}, \emph{A New Asymptotic Preserving
Scheme Based on Micro-Macro Formulation for Linear Kinetic Equations in the Diffusion Limit},
SIAM J. Sci. Comp. \textbf{31}-1 (2008), 334-368.

\bibitem{Mouton} \textsc{Mouton, A.}, \emph{Two-scale semi-lagrangian simulation of a charged particles beam in a periodic focusing channel}, Kinet. Relat. Models \textbf{2}-2 (2009), 251-274.

\bibitem{Mouton_PhD} \textsc{Mouton, A.}, \emph{Approximation multi-\'echelles de l'\'equation de Vlasov}, Th\`ese de l'Universit\'e de Strasbourg, \url{http://tel.archives-ouvertes.fr/tel-00411964/fr/} (2009).

\bibitem{Nguetseng} \textsc{Nguetseng, G.}, \emph{A general convergence result for a functional related to the theory of homogenization}, SIAM J. Math. Anal. \textbf{20}-3 (1989), 608-623.

\end{thebibliography}
\end{document}